\documentclass[11pt,reqno]{amsart}
\usepackage[margin=2.5cm]{geometry}
\usepackage[english]{babel} %Remark: use the same language for uvamath and babel
\usepackage{graphicx}
\usepackage{lipsum}

\usepackage[svgnames]{xcolor}
\usepackage{biblatex}
\usepackage[colorlinks=true,
            citecolor=Fuchsia,
            linkcolor=Fuchsia, 
            urlcolor=Fuchsia
            ]{hyperref}
\usepackage{graphicx}
\usepackage{enumerate}
\usepackage{enumitem}
\usepackage{amsmath}
\usepackage{amsthm}
\usepackage{amsfonts}
\usepackage{amssymb}
\usepackage{microtype}
\usepackage{mathtools}
\usepackage{stmaryrd}
\usepackage{bm}
\usepackage{bussproofs}
\usepackage{xfrac}
\usepackage{tikz,subfigure}
\usepackage[all,cmtip,2cell]{xy}
\usepackage{tikz}
\newtheorem{thm}{Theorem}

\newtheorem{prop}[thm]{Proposition}
\newtheorem{lemma}[thm]{Lemma}
\newtheorem{corr}[thm]{Corollary}
\theoremstyle{definition}
\newtheorem{defi}[thm]{Definition}
\theoremstyle{remark}
\newtheorem{rem}[thm]{Remark}
\newtheorem{term}[thm]{Terminology}
\newtheorem{ex}[thm]{Example}

\allowdisplaybreaks
\usepackage{hyperref}
\usetikzlibrary{cd}
\usepackage{comment}
\usepackage{multirow}
\usepackage{quiver}

\newcommand{\cd}[2][]{\vcenter{\hbox{\xymatrix#1{#2}}}}

\newcommand{\J}{\mathbb{J}}

\newcommand{\Lcoalg}{\mathsf{L}\text{-}\mathbf{Coalg}}
\newcommand{\LCoalg}{\mathsf{L}\text{-}\mathbb{C}\mathbf{oalg}}
\newcommand{\Ralg}{\mathsf{R}\text{-}\mathbf{Alg}} 
\newcommand{\RAlg}{\mathsf{R}\text{-}\mathbb{A}\mathbf{lg}}

\newcommand{\dpf}{{\pitchfork\mskip-9.7mu\pitchfork}}

\newcommand{\Talg}{\mathsf{T}\text{-}\mathbf{Alg}}

\newcommand{\Tealg}{T\text{-}\mathbf{Alg}}
\newcommand{\Toalg}{T_1\text{-}\mathbf{Alg}}
\newcommand{\Ttwalg}{T_2\text{-}\mathbf{Alg}}

\newcommand{\Tsalg}{T_1\text{-}\mathbf{Alg}_\infty^s}

\newcommand{\Ord}{Ord}

\DeclareMathOperator{\Sets}{\mathbf{Sets}}
\DeclareMathOperator{\id}{id}

\DeclareMathOperator{\Set}{Set}
\DeclareMathOperator{\colim}{colim}

\DeclareMathOperator{\Cat}{\mathbf{Cat}}
\DeclareMathOperator{\op}{\mathrm{op}}

\DeclareMathOperator{\Sq}{\mathbb{S}\mathbf{q}}

\DeclareMathOperator{\sSet}{\mathbf{sSet}}

\DeclareMathOperator{\2b}{\mathbf{2}}

\newcommand{\dcat}[1]{{\mathbb #1}}

\usepackage{xcolor}

\newcommand\justeq[1]{\mathrel{\overset{\makebox[0pt]{\mbox{\normalfont\tiny\sffamily #1}}}{=}}}

\begin{document}

\title[Small object argument]{A constructive approach to the double-categorical small object argument}
\author{Benno van den Berg}
\address{Institute for Logic, Language and Computation, Universiteit van Amsterdam, P.O. Box 94242, NL, 1090 GE Amsterdam, The Netherlands}
\email{b.vandenberg3@uva.nl}

\author{John Bourke}
\address{Department of Mathematics and Statistics, Masaryk University, Kotl\'a\v rsk\'a 2, Brno 61137, Czech Republic}
\email{bourkej@math.muni.cz}

\author{Paul Seip}
\address{Institute for Application-Oriented Knowledge Processing, Johannes Kepler University Linz, Altenberger Stra{ß}e 69, 4040 Linz, Austria}
\email{pseip@faw.jku.at}

\date{\today}

\begin{abstract}
Bourke and Garner described how to cofibrantly generate algebraic weak factorisation systems by a small double category of morphisms.  However they did not give an explicit construction of the resulting factorisations as in the classical small object argument.  In this paper we give such an explicit construction, as the colimit of a chain, which makes the result applicable in constructive settings; in particular,
our methods provide a constructive proof that the effective Kan fibrations introduced by Van den Berg and Faber appear as the right class of an algebraic weak factorisation system.
\end{abstract}

\maketitle 

\section{Introduction}

\subsection{Background and motivation}

The small object argument was  introduced by Quillen in \cite{Quillen1967} as a tool for the construction of cofibrantly generated model categories.  More generally, it is the key tool for constructing weak factorisation systems on a category $\mathcal C$, cofibrantly generated by a set of morphisms $\mathcal{J}$.  The construction itself is quite simple.  It starts by taking an arrow $f$, and factors it $f = Tf \circ Kf \colon X \to Sf \to Y$ where the object $Sf$ is constructed as a certain pushout.  By construction, the new morphism $Tf$ has liftings for those squares factoring through $f$
\begin{equation}
\xymatrix{
    A \ar[d]_{j \in \mathcal{J}} \ar[r]^{u} & X \ar[d]^(.6){f} \ar[r]^{Kf} & Sf \ar[d]^{Tf} \\
    B \ar@{-->}[urr]^<<<<<<<<<{\exists } \ar[r]_{v} & Y \ar[r]_{1} & Y
    }
\end{equation}
but not got all of the required squares.   In order to rectify this problem, one repeats the process with input $Tf$ transfinitely, obtaining a chain in the category of arrows.

\begin{equation*}
\xymatrix{
    f \ar[rr]^{\eta_f = (Kf,1)} && Tf \ar[rr]^{\eta_{Tf}} && T^2f \ldots \ar[rr] && Rf \ldots
    }
\end{equation*}

Under mild size assumptions on $\mathcal{J}$, eventually the transfinite composite $Rf$ obtains all fillers, and one obtains the desired factorisation $f = Rf \circ Lf \colon X \to Ef \to Y$ for the weak factorisation system.

In his groundbreaking paper \cite{Garner2011Understanding}, Garner observed that the small object argument has some deficiencies, because it adds redundant information.  In particular, the object $T^2f$ in the above chain has two fillers against those squares factoring through $Tf$.  This redundancy can be factored out by replacing the second term $T^2 f$ in the above sequence by the coequaliser
\[
        \xymatrix{
        Tf \ar@<1ex>[rr]^{T\eta_{f}} 
        \ar@<-1ex>[rr]_{\eta_{Tf}} & & T^{2}f \ar[r]^{} & T_2 f \\
        }
 \]
and continuing in a similar fashion to obtain a new sequence

\begin{equation*}
\xymatrix{
    f \ar[rr]^{\eta_f } && T f \ar[rr]^{} && T_2 f \ldots \ar[rr] && Rf \ldots
    }
\end{equation*}

This new sequence is, in fact, the free algebra sequence for the pointed endofunctor $(T, \eta)$.  Furthermore, under mild conditions on $\mathcal{J}$ it converges, producing the free $T$-algebra on $f$.  This is a morphism $Rf$ equipped with a \emph{canonical filler} against each square with left leg in $\mathcal{J}$, as depicted below.

\begin{equation*}
\xymatrix{
    A \ar[d]_{j \in \mathcal{J}} \ar[rr]^{u} && Ef \ar[d]^{Rf} \\
    B \ar@{-->}[urr]^{\phi_j(u,v) } \ar[rr]_{v} && Y 
    }
\end{equation*}

Using this construction, one obtains a factorisation $f = Rf \circ Lf$ as before.  However, the result is not merely a weak factorisation system, but a so-called \emph{algebraic weak factorisation system} (awfs) \cite{Grandis2006}.  Algebraic weak factorisation systems refine weak factorisation systems by allowing morphisms equipped with structure (a choice of liftings) as opposed to morphisms satisfying the lifting property.  

In fact, Garner's small object argument allows $\mathcal J$ to be a small category of morphisms --- that is, a category equipped with a functor $\mathcal{J} \to \mathcal{C}^{\2b}$ to the category of arrows in $\mathcal C$.  Then the resulting $T$-algebras are morphisms $f$ equipped with a lifting function $\phi$ satisfying the \emph{horizontal} compatibility 

 \[
\cd{A \ar[d]_{i} \ar[r]^-{s}  & C \ar[d]_{j} \ar[r]^u & X  \ar[d]^{f} \\ 
B \ar[r]_-{t}  & D \ar@{.>}[ur]|{\phi_{j}(u,v)} \ar[r]_v & Y} 
  \quad = \quad 
  \cd{A \ar[d]_{i} \ar[rr]^-{u.s}  && X \ar[d]^{f} \\ 
  B \ar@{.>}[urr]|{\phi_{i}(u.s,v.t)} \ar[rr]_-{v.t}  && Y}
\]
for $(s,t) \colon i \to j$ a morphism of  $\mathcal{J}$.
  
Later, Bourke and Garner \cite{Bourke2016Accessible} further generalised this, allowing now $\mathcal{J}$ to be a small double category of morphisms --- the idea here being that one can ask that the liftings satisfy \emph{vertical} compatibilities in addition to horizontal ones.  This greater generality allows for a host of new examples and was shown to be best possible in the sense that each accessible awfs on a locally presentable category is cofibrantly generated by a small double category of morphisms.

Unlike in the earlier paper of Garner however, Bourke and Garner did not give an explicit construction of the awfs generated by a small double category and also assumed the base category to be locally presentable.\begin{footnote}{They established its existence as a coequaliser of accessible monads (themselves free on accessible pointed endofunctors) such being known to exist by earlier results of Kelly \cite{Kelly1980A-unified}.}\end{footnote}   Our goal in the present paper is to give such an explicit construction, avoiding the need for local presentability.   In fact, we will see that it is a natural enhancement of the small object argument of Garner described above, obtained by replacing the coequaliser at each successor stage by a joint coequaliser (in order to impose the vertical compatibility conditions).  The main result is Theorem \ref{thm:main} and a comparison with the small object argument of Garner described above is given in Remark \ref{rem:explicit}.

Our main motivation for giving an explicit construction for awfs cofibrantly generated by small double categories comes from the work of Van den Berg and coauthors \cite{BergFaber2022,BergGeerligs25,Tanaka25} on simplicial homotopy theory in the constructive setting.  This involves the notion of an effective Kan fibration, which can be described using double categorical lifting properties, but in order to obtain a constructive algebraic model structure for simplicial sets, they require giving a constructive small object argument.  The constructive small object argument presented in Theorem \ref{thm:constructivemain} here solves that problem.

\subsection{Related work}

Our results build on the thesis \cite{Seip2024} of the third-named author Seip, under the supervision of van den Berg.  The present paper improves the results of \cite{Seip2024} by giving an explicit construction for the small object argument and more general hypotheses.

Earlier work analysing the small object argument from a constructive point of view has exclusively focused on cofibrant generation by a small category of morphisms. In particular, In \cite{swan18} Andrew Swan gave a constructive small object argument for small categories of morphisms using a weak choice axiom called WISC. In \cite{cavallosattler25} (see also \cite{sattler25}) the authors also go over this argument in an attempt to get a better understanding of the left class of a cofibrantly generated awfs; for this they also focus on the setting of categories of morphisms and rely more on the work of \cite{Kelly1980A-unified} than \cite{Koubek1979Categorical}, as we will do. In addition, Garner's account of his small object argument for categories of morphisms has been formalised in the Unimath library of the Rocq proof assistant by Hilhorst and North \cite{hilhorstnorth24}; their proof also only uses constructive principles and focuses on the finitary case. 

\subsection{Contents}

Let us now give an outline of the paper.  In Section \ref{sect:prelim}, we recall a small amount of background about cofibrantly generated algebraic weak factorisation systems and free algebras on pointed endofunctors.

Section \ref{sect:categorical} is devoted to obtaining a better understanding of the pointed endofunctor $T$ arising in the first stage of Garner's small object argument.  In particular we establish a new universal property of $T$ with respect to $1$-step lifting operations.  We also revisit Garner's small object argument for the awfs generated by a small category of morphisms.  

In Section \ref{sect:double} we prove our main results, Theorem \ref{thm:main} and Theorem \ref{thm:constructivemain}, on cofibrant generation by small double categories.

%In Section 4 we reformulate double-categorical lifting operations using special algebras for a pair of pointed endofunctors and in Section 5 describe the construction of such free special algebras.  Finally, in Section 6 we combine the work of the preceding sections to give our main cofibrant generation result in Theorem \ref{thm:main}.

%The core of the paper is Section 4.  Here we describe the free construction of special algebraic chains, which leads to an explicit small object argument for the awfs cofibrantly generated by a small double category.

\section{Preliminaries}\label{sect:prelim}

In this section, we review some background material on algebraic weak factorisation systems before recalling the construction of free algebras on pointed endofunctors, which is crucial in the construction of cofibrantly generated awfs.

\subsection{Cofibrantly generated algebraic weak factorisation systems}

Algebraic weak factorisation systems (awfs) were introduced by Grandis and Tholen \cite{Grandis2006}.  The definition of awfs was refined by Garner \cite{Garner2011Understanding} and their basic theory was further developed in \cite{Riehl2011Algebraic, Bourke2016Accessible, Bourke2023Orthogonal}.  In this section, we quickly recall awfs, double-categorical liftings and cofibrant generation of awfs, primarily following \cite{Bourke2016Accessible} and \cite{Bourke2023Orthogonal}.

To begin with, an awfs $(L,R)$ on $\mathcal C$ consists of a comonad $L$ and monad $R$ on the the category of arrows $\mathcal{C}^{\2b}$ satisfying various compatibilities.  (For the full definition, not needed here, see for instance \cite{Garner2011Understanding}.)  The categories $ \Lcoalg \to \mathcal{C}^{\2b}$ and $ \Ralg \to \mathcal{C}^{\2b}$ of coalgebras and algebras, equipped with their forgetful functors to the category of arrows, are thought of as the categories of left and right maps of the awfs.  Their objects are pairs $(f,\phi)$ where $f \colon A \to B$ is a morphism of $\mathcal C$ and $\phi$ the additional (co)algebra structure, whilst morphisms $(u,v) \colon (f,\phi) \to (g,\theta)$ in these categories are commutative squares commuting with the additional (co)algebra structure.
\begin{equation*}
\xymatrix{
A\ar[d]_{(f,\phi)} \ar[r]^{u} & C \ar[d]^{(g,\theta)} \\
B \ar[r]_{v} & D
}
\end{equation*}
In fact, squares such as the above one are the squares in double categories $\LCoalg$ and $\RAlg$.  The key additional operation is that both $L$-coalgebras and $R$-algebras can be composed (vertically) and this enhances $\Lcoalg$ and $\Ralg$ to double categories of left and right maps $\LCoalg$ and $\RAlg$, equipped with forgetful double functors $\LCoalg \to \Sq(\mathcal{C})$ and $\RAlg \to \Sq(\mathcal{C})$ to the double category of commutative squares in $\mathcal C$.

In fact, by Proposition 20 of \cite{Bourke2016Accessible}, the whole awfs is determined up to isomorphism by either of these double categories over $\Sq(\mathcal{C})$, similar to the fact that a weak factorisation system is determined by its left or right class.  For instance, we have $\LCoalg^\dpf \cong \RAlg$, where $(-)^\dpf$ is the right lifting operation construction, which we now recall.

This operation $(-)^\dpf$  takes as input a double functor $U \colon \mathbb{J} \to \Sq(\mathcal{C})$ and produces a further double category equipped with a forgetful double functor $V \colon \mathbb{J}^\dpf \to \Sq(\mathcal{C})$.    %Consider a functor $U : \mathcal{J} \to \mathcal{C}^{\2b}$.  We recall that this gives rise to a double category $\mathbb{J}^\dpf$ equipped with a forgetful double functor  $\mathbb{J}^\dpf \to \Sq(\mathcal{C})$ to the double category of commutative squares in $\mathcal{C}$.
The objects and horizontal arrows of $\mathbb{J}^\dpf$ are those of $\mathcal{C}$ itself.  A vertical arrow of $\mathbb{J}^\dpf$ is a pair $(f,\phi)$ consisting of an arrow $f \colon X \to Y$ of $\mathcal{C}$ together with a lifting operation $\phi$ which provides fillers in each commutative square as below
\begin{equation}\label{eq:lifting}
\cd{ UA \ar[d]_{Uj} \ar[r]^-{u} &  X \ar[d]^{f}  \\
UB \ar@{.>}[ur]|{\phi_{j}(u,v)} \ar[r]^-{v}  &Y}
\end{equation}

The liftings must satisfy a \emph{horizontal compatibility condition}, which says that given a morphism $r\colon i \to j \in \mathcal{J}_1$, we have the equality of diagonals in 

\begin{equation}\label{eq:horizontal}
\cd{UA \ar[d]_{Ui} \ar[r]^-{Ur_0}  & UC \ar[d]_{Uj} \ar[r]^s & X  \ar[d]^{f} \\ 
UB \ar[r]_-{Ur_1}  & UD \ar@{.>}[ur]|{\phi_{j}} \ar[r]_t & Y} 
  \quad = \quad 
  \cd{UA \ar[d]_{Ui} \ar[rr]^-{s.Ur_0}  && X \ar[d]^{f} \\ 
  UB \ar@{.>}[urr]|{\phi_{i}} \ar[rr]_-{t.Ur_1}  && Y}
\end{equation}
where we have omitted certain labels for $\phi$.  Furthermore, they must satisfy a \emph{vertical compatibility condition}, which says that given a composable pair $j \circ i\colon A \to B \to C \in \dcat{\mathcal{J}}$ of vertical morphisms, we have the equality of diagonals  from bottom left to top right in
\begin{equation}\label{eq:vertical}
\cd[@-0.6em@C+1.2em]{
  UA \ar[d]_{Ui} \ar[rr]^s && X \ar[dd]^{f} \\
  UB \ar[d]_{Uj}  \ar@{.>}[urr]^{{\phi_{i}}}\\
  UC \ar[rr]_-{t}  \ar@{.>}[uurr]_{{\phi_{j}}} &&  Y }
  \quad = \quad
  \cd[@-0.6em@C+1.2em]{
  UA \ar[d]_{Ui} \ar[rr]^s && X \ar[dd]^{f} \\
  UB \ar[d]_{Uj} \\
  UC \ar[rr]_-{t}  \ar@{.>}[uurr]_{{\phi_{j \circ i}}} && Y }
\end{equation}
Squares in $\mathbb{J}^\dpf$ are commutative squares which are compatible with the lifting functions.  

We refer the reader to  \cite{Bourke2016Accessible} for the remaining details on the double category structure, where the evident vertical and horizontal composition operations are fully described.  For our purposes, it suffices to understand the category $\mathbb{J}^\dpf_1$ of vertical arrows and squares, in which composition is simply composition of commutative squares.  Here the faithful forgetful functor $V_1 \colon \mathbb{J}^\dpf_1 \to \mathcal{C}^{\2b}$ simply forgets the chosen liftings.

An awfs $(L,R)$ is then said to be cofibrantly generated by $\mathbb{J}$ as above if there is an isomorphism of double categories $\RAlg \cong \mathbb{J}^\dpf$ over $\Sq(\mathcal{C})$.  An important result for us is the following one.

\begin{thm} [\emph{Proposition 13 of  \cite{Bourke2023Orthogonal}}]\label{thm:doublecof}
$U \colon \mathbb{J} \to \Sq(\mathcal{C})$ cofibrantly generates an awfs $(L,R)$  if and only if $V_1 \colon \mathbb{J}^\dpf_1 \to \mathcal{C}^{\2b}$ has a left adjoint.
\end{thm}

By this result, constructing the algebraic weak factorisation systems cofibrantly generated by $\mathbb{J}$ amounts to describing the left adjoint to $V_1$, which is what we will do using an explicit construction in Section~\ref{sect:double}.

Let us also recall the simpler case of cofibrant generation by a category of morphisms $U \colon \mathcal{J} \to \mathcal{C}^{\2b}$.  This gives rise to a double category of morphisms $V \colon \mathcal{J}^\dpf \to \Sq(\mathcal{C})$ whose objects are again morphisms $f \colon X \to Y$ equipped with a lifting operation as depicted below

\begin{equation}\label{eq:lifting2}
\cd{ UA \ar[d]_{Uj} \ar[r]^-{u} &  X \ar[d]^{f}  \\
UB \ar@{.>}[ur]|{\phi_{j}(u,v)} \ar[r]^-{v}  &Y}
\end{equation}
subject to the horizontal compatiblity equation \eqref{eq:horizontal} with respect to morphisms in $\mathcal{J}$ but no vertical compatibilities.  An awfs $(L,R)$ is then said to be cofibrantly generated by $\mathcal J$ as above if there is an isomorphism of double categories $\RAlg \cong \mathcal J^\dpf$ over $\Sq(\mathcal{C})$.  

Note that the above construction on categories of morphisms is subsumed by the double-categorical version.  To see this, observe that $U$ can be extended to a morphism of internal graphs
% https://q.uiver.app/#q=WzAsNCxbMCwwLCJcXG1hdGhjYWwgSiJdLFswLDEsIlxcbWF0aGNhbCBDIl0sWzEsMCwiXFxtYXRoY2FsIENee1xcMmJ9Il0sWzEsMSwiXFxtYXRoY2FsIEMiXSxbMCwxLCJkLlUiLDIseyJvZmZzZXQiOjJ9XSxbMCwxLCJjLlUiLDAseyJvZmZzZXQiOi0yfV0sWzAsMiwiVSJdLFsyLDMsImMiLDAseyJvZmZzZXQiOi0yfV0sWzIsMywiZCIsMix7Im9mZnNldCI6Mn1dLFsxLDMsIjEiXV0=
\[\begin{tikzcd}
	{\mathcal J} & {\mathcal C^{\2b}} \\
	{\mathcal C} & {\mathcal C}
	\arrow["U", from=1-1, to=1-2]
	\arrow["{d.U}"', shift right=2, from=1-1, to=2-1]
	\arrow["{c.U}", shift left=2, from=1-1, to=2-1]
	\arrow["c", shift left=2, from=1-2, to=2-2]
	\arrow["d"', shift right=2, from=1-2, to=2-2]
	\arrow["1", from=2-1, to=2-2]
\end{tikzcd}\]
in $\Cat$, where $d$ and $c$ denote the domain and codomain operations.  Now we form the free internal category $D(\mathcal J)$ on the internal graph $d.U, c.U \colon \mathcal J \to \mathcal C$, which is the double category whose vertical morphisms (resp. squares) are composable sequences of vertical morphisms (resp. squares).  Then taking the corresponding adjoint double functor $U \colon D(\mathcal J) \to \Sq(\mathcal{C})$, we obtain $\mathcal{J}^\dpf \cong D(\mathcal{J})^\dpf$ since the vertical compatibility condition ensures that a lifting operation on a composite is constructed from the liftings against its individual components.  

In particular, cofibrant generation by categories of morphisms is a special case of cofibrant generation by double categories.

\subsection{Pointed endofunctors and their free algebras} \label{section:free_alg}

In this section we review the construction of free algebras on pointed endofunctors.  The standard reference is Kelly's paper \cite{Kelly1980A-unified}, specifically Theorems 14.3 and 15.6.  Here we instead use the approach of algebraic chains, developed by Koubek and Reiterman in \cite{Koubek1979Categorical} to describe the free algebra on an endofunctor, and slightly adapted in \cite{BourkeEquipping} to handle pointed endofunctors.  The advantage of this approach is that it emphasises the explicit formulae involved by focusing not only on the free algebra but also on the free algebraic chain.  

The construction we describe closely follows Appendix A of \cite{BourkeEquipping}.  However, we also fill in some details left to the reader in  \cite{BourkeEquipping}, since we will need to adapt these results to handle the double-categorical small object argument, which lives outside the pointed endofunctor setting.

Let $(T,\eta)$ be a pointed endofunctor on a category $\mathcal C$.  A $T$-algebra $(X,x)$ consists of an object $X \in \mathcal C$ together with a morphism $x \colon TX \to X$ such that $x \circ \eta_X = \id_X$.  Together with the evident structure-preserving morphisms, these form a category $\Talg$.  The construction of free $T$-algebras is non-trivial and involves a certain transfinite sequence. In order to explain where the sequence comes from, it is natural to use \emph{algebraic chains}.

To begin with, recall that a chain is a functor $X:\Ord \to \mathcal C$ on the posetal category of ordinals, whilst a chain map is a natural transformation.   Given a pointed endofunctor $(T,\eta)$ on $\mathcal C$ an algebraic chain $(X,x)$ is a chain $X$ together with, for each ordinal $n$, a map $x_n:TX_n \to X_{n+1}$ satisfying
\begin{itemize}
\item for all $n$
\begin{equation}\label{eq:unit}
\xymatrix{
X_{n} \ar[drr]_{j_{n}^{n+1}} \ar[rr]^{\eta_{X_{n}}} && TX_{n} \ar[d]^{x_{n}} \\
 && X_{n+1}
}
\end{equation}
%\begin{equation}\label{eq:unit}{x_{n} \circ \eta_{X_{n}} = j_{n}^{n+1}:X_{n} \to X_{n+1}} \hspace{1cm} and
%\end{equation}
\item and for all $n < m$ the diagram 
\begin{equation}\label{eq:alg}
\xymatrix{
TX_{n} \ar[d]_{x_{n}} \ar[rr]^{T(j_{n}^{m})} && TX_{m} \ar[d]^{x_{m}} \\
X_{n+1} \ar[rr]_{j_{n+1}^{m+1}} && X_{m+1}
}
\end{equation}
commutes.
\end{itemize}
A morphism $f:(X,x) \to (Y,y)$ of algebraic chains is a chain map that commutes with the $x_n$ and $y_n$ for all $n$.  These are the morphisms of the category  $\Talg_\infty$ of algebraic chains.

There is a constant functor $\Delta:\Talg \to \Talg_\infty$, sending $(X,x)$ to the constant chain on $X$ equipped with $x_n=x$ for all $n$, and a forgetful functor $V:\Talg_\infty \to \mathcal C$ sending $(X,x)$ to $X_0$ and their composite yields the forgetful functor from the categories of algebras, as depicted below.
\begin{equation}\label{eq:forgetfultriangle}
\xymatrix{
\Talg \ar[dr]_{U} \ar[r]^{\Delta} & \Talg_\infty \ar[d]^{V} \\
& \mathcal C}
\end{equation}
This breaks the problem of constructing the free algebra into two parts: (1) constructing the free algebraic chain $X^{\bullet}$ on  $X$ and (2) establishing when it admits a reflection along $\Delta$.  

The interesting part is constructing the free algebraic chain.  To see how the formula for it naturally arises, note that the equation \eqref{eq:alg} holds for all $n < m$ if it does so in the cases (a) $m=n+1$ and (b) $m$ is a limit ordinal.  Now consider a chain $X$ equipped with maps $x_{n}:TX_n \to X_{n+1}$ satisfying \eqref{eq:unit}.  Then case (a) of \eqref{eq:alg} becomes the assertion that for all $n$ the diagram
\begin{equation}\label{eq:alg1}
\xymatrix{
TX_{n} \ar@<.6ex>[rr]^{Tx_{n} \circ T\eta_{X_{n}}}\ar@<-.6ex>[rr]_{Tx_{n} \circ \eta_{TX_{n}}} && TX_{n+1} \ar[r]^{x_{n+1}} & X_{n+2}
}
\end{equation}
is a fork.   

Also, using that $x_{m} \circ \eta_{X_{m}} = j_{m}^{m+1}$, case (b) of \eqref{eq:alg} becomes the assertion that for all limit ordinals $m$ and $n<m$ the diagram
\begin{equation*}
\xymatrix{
TX_{n} \ar@<0.5ex>[rr]^{ T{j_n^m}} \ar@<-0.5ex>[rr]_{ \eta_{X_{m}} \circ j_{n+1}^{m} \circ x_{n}} && TX_{m} \ar[r]^{x_{m}} & X_{m+1}
}
\end{equation*}
is a fork, which, in the presence of colimits of chains, equally asserts that for each limit ordinal $m$ the diagram 
\begin{equation}\label{eq:alg2}
\xymatrix{
\colim_{n<m}TX_{n} \ar@<0.5ex>[rr]^(.6){\langle T{j_n^m} \rangle} \ar@<-0.5ex>[rr]_(.6){\langle \eta_{X_{m}} \circ j_{n+1}^{m} \circ x_{n} \rangle} && TX_{m} \ar[r]^{x_{m}} & X_{m+1}
}
\end{equation}
is a fork.  These reformulations lead naturally to the following proposition, whose proof is then completely routine.

\begin{prop}\label{prop:free}
If $\mathcal C$ is cocomplete then $V$ has a left adjoint whose value at $X \in \mathcal C$ is the algebraic chain $X^{\bullet}$ with values:
\begin{itemize}
\item $X^{\bullet}_{0}=X$, $X^{\bullet}_{1}=TX$, $j_{0}^{1}=\eta_{X}:X \to TX$ and $x_0=1:TX \to TX$.
\item At an ordinal of the form $n+2$ the object $X^{\bullet}_{n+2}$ is the coequaliser
\begin{equation*}
\xymatrix{
TX^{\bullet}_{n} \ar@<.6ex>[rr]^{Tx_{n} \circ T\eta_{X^{\bullet}_{n}}}\ar@<-.6ex>[rr]_{Tx_{n} \circ \eta_{TX^{\bullet}_{n}}} && TX^{\bullet}_{n+1} \ar[r]^{x_{n+1}} & X^{\bullet}_{n+2}
}
\end{equation*}
with $j_{n+1}^{n+2}=x_{n+1} \circ \eta_{X^{\bullet}_{n+1}}$.
\item At a limit ordinal $\alpha$,
\begin{itemize} 
\item $X^{\bullet}_{\alpha} = \colim_{n < \alpha} X^{\bullet}_{n}$ with the connecting maps $j_n^\alpha$ the colimit inclusions.  
\item $X^{\bullet}_{\alpha + 1}$ is the coequaliser
\begin{equation*}
\xymatrix{
\colim_{n<\alpha}TX^{\bullet}_{n} \ar@<0.5ex>[rr]^(.6){\langle T{j_n^\alpha} \rangle} \ar@<-0.5ex>[rr]_(.6){\langle \eta_{X^{\bullet}_{\alpha}} \circ j_{\alpha +1}^{\alpha} \circ x_{n} \rangle} && TX^{\bullet}_{\alpha} \ar[r]^{x_{\alpha}} & X^{\bullet}_{\alpha+1}
}
\end{equation*}
with $j_{\alpha}^{\alpha+1}=x_{\alpha} \circ \eta_{X^{\bullet}_{\alpha}}$.
\end{itemize}
\end{itemize}
\end{prop}
\begin{proof}

To see that $X^\bullet$ is an algebraic chain, observe that the unit compatibility \eqref{eq:unit} holds by definition of the connecting maps $j_n^\alpha$.  Therefore, it will be an algebraic chain just when the equations $\eqref{eq:alg1}$ and $\eqref{eq:alg2}$ hold, which is the case by definition of $X^{\bullet}$.

The unit of the adjunction will be the identity --- so, we are to show that given $f:X \to Y_{0}=V(Y,y)$ there exists a unique map $f:X_\bullet \to (Y,y)$ of algebraic chains with $f_{0}=f$.  The required commutativity below left
\begin{equation*}
\xymatrix{
TX \ar[d]_{Tf} \ar[r]^{x_{0}=1} & TX \ar[d]^{f_1} \\
TY_0 \ar[r]^{y_0} & Y_{1}}
\hspace{1cm}
\xymatrix{
TX^{\bullet}_{n} \ar[d]^{Tf_{n}} \ar@<.6ex>[rr]^{Tx_{n} \circ T\eta_{X^{\bullet}_{n}}}\ar@<-.6ex>[rr]_{Tx_{n} \circ \eta_{TX^{\bullet}_{n}}} && TX^{\bullet}_{n+1} \ar[d]^{Tf_{n+1}} \ar@{.>}[r]^{x_{n+1}} & X^{\bullet}_{n+2} \ar@{.>}[d]^{f_{n+2}} \\
TY_{n} \ar@<.6ex>[rr]^{Ty_{n} \circ T\eta_{Y_{n}}}\ar@<-.6ex>[rr]_{Ty_{n} \circ \eta_{TY_{n}}} && TY_{n+1} \ar[r]^{y_{n+1}} & Y_{n+2}
}
\end{equation*}
forces us to set $f_{1}=y_{0} \circ Tf$.  The map $f_{n+2}$ must render the right square in the diagram above right commutative.  But since the two back squares serially commute and the bottom row is a fork there exists a unique map from the coequaliser $X^{\bullet}_{n+2}$ rendering the right square commutative.  This uniquely specifies $f_\alpha$ for ordinals $\alpha$ of the form $n+2$.  At a limit ordinal $\alpha$, $f_\alpha:X^{\bullet}_{\alpha}=\colim_{n < \alpha} X^{\bullet}_{n} \to Y_{\alpha}$ is the unique map from the colimit commuting with the connecting maps --- which it must do to form a morphism of chains.  At the successor of a limit ordinal $\alpha$ there is a unique map $f_{\alpha+1}:X^{\bullet}_{\alpha+1} \to Y_{\alpha+1}$ from the coequaliser satisfying $f_{\alpha+1} \circ x_{\alpha} = y_{\alpha} \circ Tf_{\alpha}$, as required.
\end{proof}

%Note that the usual forgetful functor\footnote{Here we write $\Tealg$ for the category of algebras for the pointed endofunctor $(T,\eta)$.} $U : \Tealg \to \mathcal{C}$ factors through the forgetful functor $V : \Tealg_\omega \to \mathcal{C}$ by the functor $\Delta : \Tealg \to \Tealg_\omega$ which sends a $T$-algebra $(X,\gamma : TX \to X)$ to the constant algebraic chain with $X_n = X$ for every $n$, and $x_n = \gamma$ for every $n$. 

We say that an algebraic chain $(X,x)$ \emph{stabilises} at an ordinal $n$ if for all $m > n$ the map $j_n^m : X_n \to X_m$ is invertible.  

\begin{prop}\label{prop:stablefree}
If the algebraic $(X,x)$ stabilises at an ordinal $n$, then $X_n$ admits the structure of a $T$-algebra for the pointed endofunctor
\begin{equation}\label{eq:structure}
\beta = (j_n^{n+1})^{-1} \circ x_n:TX_n \to X_{n+1} \cong X_n
\end{equation}
and this is a reflection of $(X,x)$ along $\Delta$. 
\end{prop}
\begin{proof}
 Firstly note that $(X_n,\beta)$ is indeed a $(T,\eta)$-algebra on $X_n$ since both triangles in the following diagram commute
    \[
    \xymatrix{
    X_n \ar[rd]_{j_n^{n+1}} \ar[rrrd]^1 \ar[d]_{\eta_{X_n}} & & \\
    TX_n \ar[r]_{x_n} & X_{n+1} \ar[rr]_{(j_n^{n+1})^{-1}} & & X_n. 
    }
    \]
     Now let $f : (X_n,\beta) \to (A,a)$ be a $T$-algebra morphism. We show that this induces a unique morphism $(X,x) \to \Delta(A,a)$ of algebraic chains with $f_n = f$. Since we have $f_n := f : X_n \to A$, the rest of the morphism is uniquely determined. Indeed, by the fact that it has to be a chain morphism all the squares in the following diagram have to commute
    \[
    \xymatrix{
    X_0 \ar[d]^{f_0} \ar[r]^{j_0^1} & X_1 \ar[d]^{f_1} \ar[r]^{j_1^2} & \ldots \ar[r] & X_n \ar[d]^{f_n} \ar[r]^{j_n^{n+1}} & X_{n+1} \ar[d]^{f_{n+1}} \ar[r]^{j_{n+1}^{n+2}} & \ldots \\
    A \ar[r]_1 & A \ar[r]_1 & \ldots \ar[r]_1 & A \ar[r]_1 & A \ar[r]_1 & \ldots
    }
    \]
    This forces us to define: 
    \[
    f_k = 
    \begin{cases}
    f_n \circ j_k^n, & k < n, \\
    f_n \circ (j_n^k)^{-1}, & k > n. 
    \end{cases}
    \]
    It remains to show that the resulting map is a morphism of algebraic chains. We distinguish three cases. In the case that $k = n$, note that the following diagram commutes since $f$ is a morphism of algebras, 
    \[
    \xymatrix{
    TX_n \ar[rd]_{\beta} \ar[d]_{x_n} \ar[rr]^{T(f_n)} & & TA \ar[d]^a \\
    X_{n+1} \ar[r]_{(j_n^{n+1})^{-1}} & X_n \ar[r]_{f_n} & A.  
    }
    \]
    so we have $f_{n+1} \circ x_n = a \circ T(f_n)$. In the case $k < n$, we look at the following diagram 
    \[
    \xymatrix{
    TX_k \ar[d]_{x_k} \ar[r]^{T(j_{k}^n)} & TX_n \ar[d]^{x_n}\ar[r]^{T(f_n)} & TA \ar[dd]^a \\
    X_{k+1} \ar[d]_1 \ar[r]_{j_{k+1}^{n+1}} & X_{n+1} \ar[d]^{(j_n^{n+1})^{-1}} \\
    X_{k+1} \ar[r]_{j_{k+1}^n} & X_n \ar[r]_{f_n} & A
    }
    \]
    Since every inner square in the diagram commutes, the outer square also commutes, which shows that $f_{k+1} \circ x_k = a \circ T(f_k)$. The case $k > n$ is analogous. 
    
    Conversely, it is easy to see from the above considerations that if $f : (X,x) \to \Delta(A,a)$ is a morphism of algebraic chains, then $f_n : X_n \to A$ is an algebra morphism $(X_n, \beta) \to (A,a)$. 
\end{proof}

\begin{term}(\emph {Size assumptions.})
We now introduce some terminology and conditions concerning size, following \cite{Koubek1979Categorical}.  Firstly, by an $\alpha$-chain we mean a functor $X \colon Ord_{< \alpha} \to \mathcal C$ from the full subcategory of ordinals less than $\alpha$.  Also, for $\alpha$ a limit ordinal, we say that $A \in \mathcal C$ is $\alpha$-small if $\mathcal C(A,-)$ preserves colimits of $\alpha$-chains.  

The simplest size condition that ensures the algebraic chain $X^{\bullet}$ stabilises is that $T$ preserves colimits of $\alpha$-chains for some limit ordinal $\alpha$ --- for instance, if $T$ preserves $\alpha$-filtered colimits.  

In order to cover non-locally presentable examples, such as topological spaces, it is useful to consider also a second size condition with respect to a suitable factorisation system $(E,M)$ on $\mathcal C$.   To this end, recall that a factorisation system $(E,M)$ on $\mathcal C$ is said to be \emph{proper} if each $E$-map is an epi and each $M$-map is a mono, and it is \emph{co-well-powered} if each object $X \in \mathcal C$ admits, up to isomorphism, only a set of $E$-quotients, where an $E$-quotient of $X$ is a morphism in $E$ with domain $X$.  On topological spaces, an important example has left class the surjective continuous maps and right class the subspace embeddings.  

In this context, an $(\alpha,M)$-chain is an $\alpha$-chain all of whose connecting homomorphisms belong to the right class $M$.  We say that $A \in \mathcal C$ is $(\alpha,M)$-small if $\mathcal C(A,-)$ preserves colimits of $(\alpha,M)$-chains.  
\end{term}

The second size condition is now stated in the theorem below.

\begin{thm}[Theorem 24 of \cite{BourkeEquipping}]\label{thm:chains}
Let $(T,\eta)$ be a pointed endofunctor on a cocomplete category $\mathcal C$.  If either
\begin{enumerate}
\item[(a)] $T$ preserves colimits of $\alpha$-chains for some limit ordinal $\alpha$, or
\item[(b)] $\mathcal C$ is equipped with a co-well-powered proper factorisation system $(E,M)$ such that $T$ preserves colimits of  $(\alpha,M)$-chains for some limit ordinal $\alpha$.
\end{enumerate}
Then (1) each algebraic chain $X^{\bullet}$ stabilises and (2) its point of stabilisation, with algebra structure as in \eqref{eq:structure}, is the free $T$-algebra on $X$.
\end{thm}
\begin{proof}
We defer the proof that $X^{\bullet}$ stabilises to the appendix.  Assuming this, since the triangle
\begin{equation}\label{eq:forgetfultriangle2}
\xymatrix{
\Talg \ar[dr]_{U} \ar[r]^{\Delta} & \Talg_\infty \ar[d]^{V} \\
& \mathcal C}
\end{equation}
commutes, the free $T$-algebra is given by the reflection of the free algebraic chain $X^{\bullet}$ of Proposition \ref{prop:free} along $\Delta$.  By Proposition \ref{prop:stablefree}, this reflection is the point of stabilisation of $X^{\bullet}$.
\end{proof}

\section{The universal property of the pointed endofunctor arising in the small object argument}\label{sect:categorical}

In this section we take a close look at the pointed endofunctor $T$ arising in the first step of Garner's small object argument.  In Theorem \ref{univ_prop_T} we describe a new universal property of $T$ with respect to \emph{$1$-step lifting operations}.  This universal property of $T$ will be of central importance in Section \ref{sect:double}.  We conclude this section by re-deriving Garner's small object argument for cofibrant generation by a small category of morphisms.  

\subsection{The pointed endofunctor and its universal property}
Let $\mathcal C$ be a cocomplete locally small category, and consider a functor $U : \mathcal{J} \to \mathcal{C}^{\2b}$ where $\mathcal {J}$ is small.  In this case, as explained in \cite{Garner2011Understanding}, there is a pointed endofunctor $T$ with $\Tealg \cong \mathcal{J}^\pitchfork$.  %Let us recall the construction of $T$ before describing its universal property in representing \emph{one-step lifting operations}, and then use this to rederive the aforementioned isomorphism.

The construction begins by forming the left Kan extension $C \colon \mathcal{C}^{\2b} \to \mathcal{C}^{\2b}$ of $U$ along itself (i.e. the density comonad of $U$).  This can be described using coends or conical colimits over comma categories, and we will use the latter description here.  

For $f \in \mathcal{C}^{\2b}$, consider the comma category $U \downarrow f$ whose objects are pairs $(j, \sigma)$, where $j \in \mathcal{J}$ is an object, and $\sigma : Uj \to f$ is a commutative square in $\mathcal{C}$ and whose morphisms $(i,\tau) \to (j,\sigma)$ are morphisms $\alpha : i \to j$ in $\mathcal{J}$ such that $\sigma \circ U \alpha = \tau$.

%
%as depicted below
%\[
%\xymatrix{
%\ar[d]_{Ui} \ar[r]^{(U\alpha)_0} \ar@/^1.5pc/[rr]^{\tau_0} & \ar[d]_{Uj} \ar[r]^{\sigma_0} & \ar[d]^f \\
%\ar@/_1.5pc/[rr]_{\tau_1} \ar[r]_{(U\alpha)_1} & \ar[r]_{\sigma_1} & 
%}
%\]

We have the forgetful functor $U \downarrow f \to \mathcal{J} : (j,\sigma) \mapsto j$ and then 
$$Cf := \colim(U \downarrow f \to \mathcal{J} \to \mathcal{C}^{\2b})$$
with colimit cocone $\iota_{j,\sigma} \colon Uj \to Cf$ for $(j,\sigma) \in \mathcal{J} \downarrow f$.

By construction, the family $(\sigma : Uj \to f)_{(j,\sigma) \in \mathcal{U} \downarrow f}$ forms a cocone over the diagram and thus we get a uniquely induced arrow $\varepsilon_f : Cf \to f$ such that 
% https://q.uiver.app/#q=WzAsMyxbMCwwLCJVaiJdLFsyLDAsIkNmIl0sWzIsMSwiZiJdLFswLDEsIlxcaW90YV97aixcXHNpZ21hfSJdLFsxLDIsIlxcZXBzaWxvbl9mIl0sWzAsMiwiXFxzaWdtYSIsMl1d
\begin{equation}
\begin{tikzcd}\label{epsilon}
	Uj && Cf \\
	&& f
	\arrow["{\iota_{j,\sigma}}", from=1-1, to=1-3]
	\arrow["\sigma"', from=1-1, to=2-3]
	\arrow["{\varepsilon_f}", from=1-3, to=2-3]
\end{tikzcd}
\end{equation}
 commutes for all $(j, \sigma) \in \mathcal{U} \downarrow f$.   We then form the pushout in the left square below, which induces a factorisation of $\varepsilon_f$ as depicted:
    \begin{equation}\label{def_T}
        \xymatrix{
    A \ar[d]_{Cf} \ar[r]^{(\varepsilon_f)_0} \ar@/^2pc/[rr]^{(\varepsilon_f)_0} & X \ar[d]_{Kf} \ar[r]^{1} &  X \ar[d]^{f}  \\
    B \ar@/_2pc/[rr]_{(\varepsilon_f)_1} \ar[r]^{q_f} & Sf \ar@{-->}[r]^{Tf} & Y.
    }
  \hspace{2cm}
    \xymatrix{
    Cf\ar[dr]_{\varepsilon_f} \ar[r]^{\gamma_f} & Kf \ar[d]^{(1,Tf)} \\
    & f}
   \end{equation}

  in which $Tf$ is the uniquely induced map from the pushout.  
  
  In the corresponding factorisation above right, we denote $\gamma_f := ((\varepsilon_f)_0,q_f)$ for convenience.  In fact, this is the factorisation associated to an orthogonal factorisation system on $\mathcal{C}^{\2b}$, whose left class consist of the pushout squares and whose right class are morphisms $\alpha \colon f \to g$ in $\mathcal{C}^{\2b}$ with $\alpha_0$ an isomorphism.  We will use this fact again shortly.

Furthermore, we obtain a commutative square 
\begin{equation}
\xymatrix{
X \ar[d]_{f} \ar[r]^{Kf} & Sf \ar[d]^{Tf}  \\
Y \ar[r]_1 & Y 
}
\end{equation}
which we denote by $\eta_f :=(Kf,1) \colon f \to Tf$.  We now turn towards describing a universal property of $Tf$.

\begin{defi}
A \emph{one-step lifting structure} from $f$ to $g$ consists of a morphism $u : f \to g \in \mathcal{C}^{\2b}$ equipped with the following lifting operation: for any $j \in \mathcal{J}$ and square $\sigma : Uj \to f$ we have a lift $\phi_{j}(\sigma)$ as shown in the following diagram 
\begin{equation} \label{one_step_lift}
\xymatrix{
    A_j \ar[d]_{Uj} \ar[r]^{\sigma_0} & X \ar[d]^(.55){f} \ar[r]^{u_0} & C \ar[d]^g \\
    B_j \ar@{-->}[urr]^<<<<<<<<<{\phi_j(\sigma)} \ar[r]_{\sigma_1} & Y \ar[r]_{u_1} & D
    }
\end{equation}
which moreover satisfies the horizontal condition: that is, for any $\alpha : i \to j$ in $\mathcal{J}$ we have $\phi_i(\sigma \circ U\alpha) = \phi_j(\sigma) \circ (U\alpha)_1$. %Intuitively, the idea is that we can solve a lifting problem for $f$ in the `next stage' $g$.
\end{defi}

\begin{ex}\label{1step}
The square $\eta_f = (Kf,1) : f \to Tf$ has a one-step lifting operation $\theta_j(\sigma) :=  q_f \circ (\iota_{j,\sigma})_1$ as depicted below
\[
\xymatrix{
A_j \ar[d]_{Uj} \ar[r]^{(\iota_{j,\sigma})_0} & A \ar[d]^{Cf} \ar[r]^{(\varepsilon_f)_0} & X \ar[d]^(.6){f} \ar[r]^{Kf} & Sf \ar[d]^{Tf} \\
B_j \ar[r]_{(\iota_{j,\sigma})_1} & B \ar@{-->}[urr]^<<<<<<<<<<{q_f} \ar[r]_{(\varepsilon_f)_1} & Y \ar[r]_{1} & Y. 
}
\]
where $q_f$ is as in diagram \eqref{def_T}.  Note that the horizontal compatibility condition follows easily from the fact that the colimit cocone components $\iota_{j,\sigma} \colon Uj \to Cf$ are natural in morphisms of $U \downarrow f$.
\end{ex}

These $1$-step lifting operations assemble naturally into a presheaf $$\mathcal{J}\textnormal{-1-Step} \colon ({\mathcal{C}^{\2b}})^{op} \times \mathcal{C}^{\2b} \to \Set$$ which sends $(f,g)$ to the set of $1$-step lifting operations from $f$ to $g$.  This comes equipped with a natural transformation $V \colon \mathcal{J}\textnormal{-1-Step} \to \mathcal{C}^{\2b}(-,-)$ which forgets the lifting operation.

\begin{thm} \label{univ_prop_T}
We have an isomorphism $\kappa_{f,g} \colon \mathcal{C}^{\2b}(Tf,g) \cong \mathcal{J}\textnormal{-1-Step}(f,g)$ natural in $g$.
\end{thm}
\begin{proof}
To prove the result, we will show that $(Kf,1) : f \to Tf$ equipped with the $1$-step lifting operation of Example~\ref{1step} is the universal such lifting operation.  That is, given $(u,\phi) \in \mathcal{J}\textnormal{-1-Step}(f,g)$, we must show that there exists a unique morphism $t : Sf \to C$ making the following diagram commute
    \begin{equation} \label{unique_map}
    \xymatrix{
    X \ar[d]_{f} \ar@/^1.5pc/[rr]^{u_0} \ar[r]_{Kf} & Sf \ar[d]_{Tf} \ar@{-->}[r]_{t} & C \ar[d]^{g} \\
    Y \ar[r]_{1} & Y \ar[r]_{u_1} & D
    }
    \end{equation} 
 and which also commutes with the lifting operations, in the sense that $t \circ \theta_j(\sigma) = \phi_j(\sigma)$ for all $j \in \mathcal{J}$ and $\sigma : Uj \to f$. 
 
 In order to prove this cleanly, let us begin by reformulating the notion of a $1$-step lifting operation as a structure internal to the category of arrows.  Firstly, given a morphism $g \colon C \to D$, observe that we have the morphism $(1_C,g) \colon 1_C \to g \in \mathcal{C}^{\2b}$, as depicted below left. 
 
% https://q.uiver.app/#q=WzAsMTgsWzAsMCwiQyJdLFswLDEsIkMiXSxbMSwwLCJDIl0sWzEsMSwiRCJdLFsyLDBdLFszLDAsIkEiXSxbMywxLCJCIl0sWzQsMSwiRCJdLFs0LDAsIkMiXSxbNiwwLCJBIl0sWzYsMSwiQiJdLFs3LDAsIkMiXSxbNywxLCJDIl0sWzgsMCwiQyJdLFs4LDEsIkQiXSxbMTAsMSwiZiJdLFsxMSwxLCJnIl0sWzExLDAsIjFfYyJdLFswLDEsIjFfQyIsMl0sWzAsMiwiMV9DIl0sWzIsMywiZyJdLFsxLDMsImciLDJdLFs1LDYsImYiLDJdLFs2LDcsInVfMSIsMl0sWzUsOCwidV8wIl0sWzgsNywiZyJdLFs2LDgsInQiLDEseyJzdHlsZSI6eyJib2R5Ijp7Im5hbWUiOiJkYXNoZWQifX19XSxbOSwxMCwiZiIsMl0sWzksMTEsInVfMCJdLFsxMSwxMiwiMV9DIl0sWzEwLDEyLCJ0IiwyLHsic3R5bGUiOnsiYm9keSI6eyJuYW1lIjoiZGFzaGVkIn19fV0sWzExLDEzLCIxX0MiXSxbMTMsMTQsImciXSxbMTIsMTQsImciLDJdLFs5LDEzLCJ1XzAiLDAseyJjdXJ2ZSI6LTR9XSxbMTAsMTQsInVfMSIsMix7ImN1cnZlIjo0fV0sWzE1LDE2LCJ1IiwyXSxbMTcsMTYsIigxX2MsZykiXSxbMTUsMTcsIih1XzAsdCkiLDAseyJzdHlsZSI6eyJib2R5Ijp7Im5hbWUiOiJkYXNoZWQifX19XV0=
\[\begin{tikzcd}
	C & C & {} & A & C && A & C & C &&& {1_C} \\
	C & D && B & D && B & C & D && f & g
	\arrow["{1_C}", from=1-1, to=1-2]
	\arrow["{1_C}"', from=1-1, to=2-1]
	\arrow["g", from=1-2, to=2-2]
	\arrow["{u_0}", from=1-4, to=1-5]
	\arrow["f"', from=1-4, to=2-4]
	\arrow["g", from=1-5, to=2-5]
	\arrow["{u_0}", from=1-7, to=1-8]
	\arrow["{u_0}", curve={height=-24pt}, from=1-7, to=1-9]
	\arrow["f"', from=1-7, to=2-7]
	\arrow["{1_C}", from=1-8, to=1-9]
	\arrow["{1_C}", from=1-8, to=2-8]
	\arrow["g", from=1-9, to=2-9]
	\arrow["{(1,g)}", from=1-12, to=2-12]
	\arrow["g"', from=2-1, to=2-2]
	\arrow["t"{description}, dashed, from=2-4, to=1-5]
	\arrow["{u_1}"', from=2-4, to=2-5]
	\arrow["t"', dashed, from=2-7, to=2-8]
	\arrow["{u_1}"', curve={height=24pt}, from=2-7, to=2-9]
	\arrow["g"', from=2-8, to=2-9]
	\arrow["{(u_0,t)}", dashed, from=2-11, to=1-12]
	\arrow["u"', from=2-11, to=2-12]
\end{tikzcd}\]

Now given $u \colon f \to g$, observe that a lifting $t$ for this square, as in the second diagram, amounts equally to a morphism $t$, as in the third diagram, making that diagram commute.  In other words, liftings are in correspondence with factorisations of $u \colon f \to g$ through $(1,g) \colon 1_C \to g$, as depicted in the fourth diagram.  (Note that a lifting as in the fourth diagram is forced to have $u_0$ in its domain component since $(1,g)$ has identity as its domain component.)

Building on this, a $1$-step lifting operation on $u \colon f \to g$ is equally specified by a lifting as on the upper horizontal in the left diagram below, natural in $j \in \mathcal{J}$.

% https://q.uiver.app/#q=WzAsMTYsWzAsMCwiVWoiXSxbMCwyLCJmIl0sWzIsMCwiMV9jIl0sWzIsMiwiZyJdLFs1LDAsIlVqIl0sWzUsMywiZiJdLFs3LDAsIjFfYyJdLFs3LDMsImciXSxbNSwxLCJDZiJdLFs1LDIsIktmIl0sWzksMCwiVWoiXSxbOSwzLCJmIl0sWzExLDAsIjFfYyJdLFsxMSwzLCJnIl0sWzksMSwiQ2YiXSxbOSwyLCJLZiJdLFswLDEsIlxcc2lnbWEiLDJdLFswLDIsIih1XzAgXFxjaXJjIFxcc2lnbWFfMCxcXHBoaV9qKFxcc2lnbWEpKSIsMCx7InN0eWxlIjp7ImJvZHkiOnsibmFtZSI6ImRhc2hlZCJ9fX1dLFsyLDMsIigxLGcpIl0sWzEsMywidSIsMl0sWzQsNiwiKHVfMCBcXGNpcmMgXFxzaWdtYV8wLFxccGhpX2ooXFxzaWdtYSkpIl0sWzYsNywiKDEsZykiXSxbNSw3LCJ1IiwyXSxbNCw4LCJcXGlvdGFfe2osXFxzaWdtYX0iXSxbOCw5LCJcXGdhbW1hX2YiXSxbOSw1LCIoMSxUZikiXSxbOSw2LCIoMikiLDEseyJzdHlsZSI6eyJib2R5Ijp7Im5hbWUiOiJkYXNoZWQifX19XSxbOCw2LCIoMSkiLDEseyJzdHlsZSI6eyJib2R5Ijp7Im5hbWUiOiJkYXNoZWQifX19XSxbNCw1LCJcXHNpZ21hIiwyLHsiY3VydmUiOjV9XSxbMTAsMTIsIih1XzAgXFxjaXJjIFxcc2lnbWFfMCxcXHBoaV9qKFxcc2lnbWEpKSJdLFsxMiwxMywiKDEsZykiXSxbMTEsMTMsInUiLDJdLFsxMCwxNCwiXFxpb3RhX3tqLFxcc2lnbWF9IiwyXSxbMTQsMTUsIlxcZ2FtbWFfZiIsMl0sWzE1LDExLCIoMSxUZikiLDJdLFsxNSwxMiwiKHVfMCx0KSIsMSx7InN0eWxlIjp7ImJvZHkiOnsibmFtZSI6ImRhc2hlZCJ9fX1dLFsxMCwzNSwiKEEpIiwxLHsic2hvcnRlbiI6eyJ0YXJnZXQiOjIwfSwic3R5bGUiOnsiYm9keSI6eyJuYW1lIjoibm9uZSJ9LCJoZWFkIjp7Im5hbWUiOiJub25lIn19fV0sWzM1LDEzLCIoQikiLDEseyJvZmZzZXQiOjQsInNob3J0ZW4iOnsic291cmNlIjoyMH0sInN0eWxlIjp7ImJvZHkiOnsibmFtZSI6Im5vbmUifSwiaGVhZCI6eyJuYW1lIjoibm9uZSJ9fX1dXQ==
\[\begin{tikzcd}
	Uj && {1_C} &&& Uj && {1_C} && Uj && {1_C} \\
	&&&&& Cf &&&& Cf \\
	f && g &&& Kf &&&& Kf \\
	&&&&& f && g && f && g
	\arrow["{(u_0 \circ \sigma_0,\phi_j(\sigma))}", dashed, from=1-1, to=1-3]
	\arrow["\sigma"', from=1-1, to=3-1]
	\arrow["{(1,g)}", from=1-3, to=3-3]
	\arrow["{(u_0 \circ \sigma_0,\phi_j(\sigma))}", from=1-6, to=1-8]
	\arrow["{\iota_{j,\sigma}}", from=1-6, to=2-6]
	\arrow["\sigma"', curve={height=30pt}, from=1-6, to=4-6]
	\arrow["{(1,g)}", from=1-8, to=4-8]
	\arrow["{(u_0 \circ \sigma_0,\phi_j(\sigma))}", from=1-10, to=1-12]
	\arrow["{\iota_{j,\sigma}}"', from=1-10, to=2-10]
	\arrow["{(1,g)}", from=1-12, to=4-12]
	\arrow["{(1)}"{description}, dashed, from=2-6, to=1-8]
	\arrow["{\gamma_f}", from=2-6, to=3-6]
	\arrow["{\gamma_f}"', from=2-10, to=3-10]
	\arrow["u"', from=3-1, to=3-3]
	\arrow["{(2)}"{description}, dashed, from=3-6, to=1-8]
	\arrow["{(1,Tf)}", from=3-6, to=4-6]
	\arrow[""{name=0, anchor=center, inner sep=0}, "{(u_0,t)}"{description}, dashed, from=3-10, to=1-12]
	\arrow["{(1,Tf)}"', from=3-10, to=4-10]
	\arrow["u"', from=4-6, to=4-8]
	\arrow["u"', from=4-10, to=4-12]
	\arrow["{(A)}"{description}, draw=none, from=1-10, to=0]
	\arrow["{(B)}"{description}, shift right=4, draw=none, from=0, to=4-12]
\end{tikzcd}\]

Turing to the central diagram, we obtain the composite left vertical by factoring $\sigma$ using \eqref{epsilon} and \eqref{def_T}.  Now observe that naturality in $\mathcal J$ of the upper horizontal morphisms $Uj  \to 1_C$ asserts that they form a cocone, whereby there exists a unique morphism $Cf \to 1_C$ as in (1) above making the upper triangle commute for each $(j,\sigma)$.  Moreover, the universal property of the colimit ensures that the quadrilateral with upper morphism (1) then also commutes, since the outside of the diagram does.  Furthermore, since $\gamma_f$ is a pushout square and $(1,g)$ has identity domain, they are orthogonal.  Therefore there exists a unique diagonal filler as in (2).   

In summary, combining the two cases of the construction so far, we conclude that there exists a unique diagonal filler as in the third diagram, making the regions (A) and $(B)$ commute for all $(j,\sigma)$.  In fact, since the lower verticals on left and right have identity domain, this diagonal must be of the form $(u_0,t)$ for some $t$ --- then the commutativity of (B) says that $t$ is a diagonal filler as below.  

% https://q.uiver.app/#q=WzAsNixbMCwwLCJYIl0sWzAsMSwiU2YiXSxbMSwwLCJYIl0sWzIsMCwiQyJdLFsyLDEsIkQiXSxbMSwxLCJCIl0sWzAsMSwiS2YiLDJdLFswLDIsIjEiXSxbMiwzLCJ1XzAiXSxbMyw0LCJnIl0sWzEsNSwiVGYiLDJdLFs1LDQsInVfMSIsMl0sWzEsMywidCIsMSx7InN0eWxlIjp7ImJvZHkiOnsibmFtZSI6ImRhc2hlZCJ9fX1dXQ==
\[\begin{tikzcd}
	X & X & C \\
	Sf & B & D
	\arrow["1_X", from=1-1, to=1-2]
	\arrow["Kf"', from=1-1, to=2-1]
	\arrow["{u_0}", from=1-2, to=1-3]
	\arrow["g", from=1-3, to=2-3]
	\arrow["t"{description}, dashed, from=2-1, to=1-3]
	\arrow["Tf"', from=2-1, to=2-2]
	\arrow["{u_1}"', from=2-2, to=2-3]
\end{tikzcd}\]

The two equalities in the above diagram are two of the three we required at the start of the proof.  Furthermore, the commutativity of the region (A) says that $t \circ q_f \circ (\iota_{j,\sigma})_1 = \phi_{j}(\sigma)$, but since $\theta_{(j,f)} = q_f \circ (\iota_{j,\sigma})_1$ by definition, this says exactly that $t \circ \theta_{(j,f)} = \phi_{j}(\sigma)$, which is the final equality required. \end{proof}

Using the naturality of the bijections $$\kappa_{f,g} \colon \mathcal{C}^{\2b}(Tf,g) \cong \mathcal{J}\textnormal{-1-Step}(f,g)$$ in $g$, the operation $T$ extends uniquely to an endofunctor in such a way that these bijections become natural in both variables.  In terms of the representations via universal $1$-step lifting operations, given  $\alpha \colon f \to g \in \mathcal{C}^{\2b}$, we define $T\alpha \colon Tf \to Tg$ to be the unique morphism such that the square

% https://q.uiver.app/#q=WzAsNCxbMCwwLCJmIl0sWzAsMSwiVGYiXSxbMiwwLCJnIl0sWzIsMSwiVGciXSxbMCwxLCJcXGV0YV9mIiwyXSxbMCwyLCJcXHNpZ21hIl0sWzIsMywiXFxldGFfZyJdLFsxLDMsIlRcXHNpZ21hIiwyXV0=
\[\begin{tikzcd}
	f && g \\
	Tf && Tg
	\arrow["\alpha", from=1-1, to=1-3]
	\arrow["{\eta_f}"', from=1-1, to=2-1]
	\arrow["{\eta_g}", from=1-3, to=2-3]
	\arrow["{T\alpha}"', from=2-1, to=2-3]
\end{tikzcd}\]
commutes and such that $T\alpha \colon Tf \to Tg$ preserves the canonical liftings --- that is, satisfies 
\begin{equation} \label{morph_prop2}
    (T\alpha)_0 \circ \theta_{j}(\sigma) = \theta_{j}(\alpha \circ \sigma).  
\end{equation}

Accordingly the components $\eta_f$ combine to give a natural transformation $\eta \colon 1 \to T$: that is, making $(T,\eta)$ into a \emph{pointed endofunctor}.  Now let $\Tealg$ denote the category of algebras for the pointed endofunctor.

\begin{prop} \label{lifting_alg}
In the above setting, we have an isomorphism $\Tealg \cong \mathcal{J}^\pitchfork$ over $\mathcal{C}^{\2b}$. %Moreover, if for every $j \in \mathcal{J}$ the domain and the codomain of $Uj$ are $\omega$-compact, then $T$ is finitary. 
\end{prop}
\begin{proof}
Note that, by construction of the natural isomorphism $\kappa$, the triangle below left commutes.
\begin{equation*}
\xymatrix{\mathcal{C}^{\2b}(T-,-) \ar[dr]_{- \circ \eta} \ar[r]^{\kappa}& \mathcal{J}\textnormal{-1-Step} \ar[d]^{V} \\
& \mathcal{C}^{\2b}(-,-)}
\hspace{2cm}
\xymatrix{T\downarrow\mathcal{C}^{\2b} \ar[dr]_{- \circ \eta} \ar[r]^--{el(\kappa)}& el(\mathcal{J}\textnormal{-1-Step}) \ar[d]^{el(V)} \\
& (\mathcal{C}^{\2b})^{\2b}}
\end{equation*}
Taking categories of elements, this therefore yields a natural isomorphism $el(\kappa)$ making the triangle above right commute where $T \downarrow \mathcal{C}^{\2b}$ is the comma category, and  $- \circ \eta \colon T\downarrow\mathcal{C}^{\2b} \to \mathcal{C}^{\2b} \downarrow \mathcal{C}^{\2b} = (\mathcal{C}^{\2b})^{\2b}$ sends $(f,\alpha \colon Tf \to g)$ to $(f,\alpha \circ \eta_f \colon f \to g)$.

Now pulling back the isomorphism $el(\kappa)$ along the identities map $I \colon \mathcal{C}^{\2b} \to (\mathcal{C}^{\2b})^{\2b} \colon f \mapsto id_f$ gives precisely the isomorphism $\Tealg \cong \mathcal{J}^\pitchfork$ over $\mathcal{C}^{\2b}$.  Indeed, an object of the pullback of $T\downarrow\mathcal{C}^{\2b}$ is a morphism $\alpha \colon Tf \to f$ for which $\alpha \circ \eta_f = 1$ --- that is, a $T$-algebra --- whilst a $1$-step lifting operation on $1 \colon f \to f$ equips $f$ with $\mathcal{J}^\pitchfork$-structure.
\end{proof}

\begin{rem}\label{rem:beta}
Explicitly, we note that given a lifting structure $(f,\phi)$, the corresponding $T$-algebra of the isomorphism in Proposition \ref{lifting_alg} is constructed as follows. By the universal property of $(Kf,1) : f \to Tf$ from Theorem \ref{univ_prop_T} we have a unique map $\beta_0: Sf \to X$ making the following diagram commute
    \begin{equation} \label{def_beta}
    \xymatrix{
        A_j \ar[d]_{Uj} \ar[r]^{\sigma_0} & X \ar[d]^{f} \ar[r]_{Kf} \ar@/^1.5pc/[rr]^{1} & Sf \ar[d]^{Tf} \ar@{-->}[r]_{\beta_0} & X \ar[d]^{f} \\
        B_j \ar[r]_{\sigma_1} & Y \ar[r]^{1} \ar@/_1.5pc/[rr]_{1_Y} & Y \ar[r]^{1_Y} & Y
        }
    \end{equation}
    and having the property that $\beta_0 \circ \theta_j(\sigma) = \phi_j(\sigma)$.   This defines the corresponding $T$-algebra $\beta = (\beta_0, 1) : Tf \to f$.  Conversely, given a $T$-algebra $(f, \beta : Tf \to f)$, the induced lifting operation is defined by $\phi_j(\sigma) := \beta_0 \circ \theta_j(\sigma)$ for a lifting problem $\sigma : Uj \to f$. (Note that because $\eta_f = (Kf,1_Y)$, the second component of any $T$-algebra structure on $f$ always equals $1_Y$; hence $\beta = (\beta_0,1_Y)$.)
\end{rem}

We now turn to the question of which colimits the endofunctor $T$ preserves.

\begin{prop} \label{prop:T_finitary}
$T$ preserves any filtered colimit which is preserved by each hom-functor $\mathcal C(UA,-)$ and $\mathcal C(UB,-)$ for $j:A \to B \in \mathcal J$.
 %   If for every $j \in \mathcal{J}$ the domain and the codomain of $Uj$ are $\lambda$-presentable, then $T$ preserves $\lambda$-filtered colimits.
\end{prop}
\begin{proof}
%Firstly we prove that $C : \mathcal{C}^{\2b} \to \mathcal{C}^{\2b}$ is $\lambda$-accessible.  
Since $C=lan_U(U)$, we have   
 \begin{equation} \label{coend_formula}
    Cf \cong \int^{j \in \mathcal{J}} \mathcal{C}^2(Uj,f) \cdot Uj, 
\end{equation}
so that $C$ is the coend of the composite functors $- \cdot Uj \circ \mathcal{C}^2(Uj,-)$.  Since coends commute with all colimits, therefore $C$ preserves any colimit preserved by each such composite.  And since the copowering functor of the second component $- \cdot Uj$ is cocontinuous, it follows that $C$ preserves any colimits preserved by each $\mathcal{C}^2(Uj,-)$.  Now by the description of filtered colimits in $\Set$, observe that $\mathcal C^{\2b}(Uj,-)$ preserves any filtered colimit preserved by each $\mathcal C(UA,-)$ and $\mathcal C(UB,-)$.  Therefore $C$ preserves any filtered colimits preserved by each $\mathcal C(UA,-)$ and $\mathcal C(UB,-)$ as above.

Turning to $T$, we first note that  since colimits are pointwise in $\mathcal{C}^{\2b}$, the domain and codomain functors $dom, cod \colon \mathcal{C}^{\2b} \to \mathcal C$ preserve and joint reflect colimits.  Therefore the two functors $dom \circ C$ and $cod \circ C$ preserve any colimits preserved by $C$.  Now by its construction in \eqref{def_T}, $S \colon \mathcal{C}^{\2b} \to \mathcal{C}$ is a pushout of these two functors and $dom$ and therefore preserves any colimits preserved by each of these three.  Hence $dom \circ T = S$ and $cod \circ T = cod$ both preserve any filtered colimit preserved by each $\mathcal C(UA,-)$ and $\mathcal C(UB,-)$, whence so does $T$.
%Firstly we prove that $C : \mathcal{C}^{\2b} \to \mathcal{C}^{\2b}$ is $\lambda$-accessible.  Indeed, as $lan_U(U)$, we have   
% \begin{equation} \label{coend_formula}
%    Cf \cong \int^{j \in \mathcal{J}} \mathcal{C}^2(Uj,f) \cdot Uj. 
%\end{equation}
%so that $C$ is the coend of the composite functor $- \cdot Uj \circ \mathcal{C}^2(Uj,-)$.  Since coends commute with all colimits, $C$ preserves any colimit preserved by each such composite; since the copowering functor of the second component $- \cdot Uj$ is cocontinuous, it follows that $C$ preserves any colimits preserved by each $\mathcal{C}^2(Uj,-)$.  Since colimits are pointwise in $\mathcal{C}^{\2b}$, the functors $dom \circ C$ and $cod \circ C$ from $\mathcal{C}^{\2b}$ to $\mathcal{C}$ both preserve   $\lambda$-filtered colimits; since $S \mathcal{C}^{\2b} \to \mathcal{C}$ is a pushout of these two and $dom$, it therefore preserves $\lambda$-filtered colimits too.  Hence $dom \circ T = S$ and $cod \circ T = cod$ both preserve $\lambda$-filtered colimits, whence so does $T$.
\end{proof}

\subsection{Cofibrant generation by a small category of morphisms}

With this in place, we can re-derive Garner's result about cofibrant generation of awfs generated by small categories of morphisms.\begin{footnote}{The assumptions in Theorem \ref{thm:categories} are less restrictive than those in Theorem 4.4 of \cite{Garner2011Understanding}, which assumes either local presentability or local boundedness.   However those in \cite{Garner2011Understanding} are simply a convenient choice of assumptions and the proof given there applies equally under the present assumptions.}\end{footnote}

\begin{thm}\label{thm:categories}[Theorem 4.4 of \cite{Garner2011Understanding}]
Let $\mathcal{C}$ be a cocomplete locally small category, $\mathcal J$ be a small category and consider $U : \mathcal J \to \mathcal{C}^{\2b}$.   If there exists a limit ordinal $\alpha$ such that either
\begin{enumerate}
%\item there exists a limit ordinal $\alpha$ such that for each $j \colon A \to B \in J$, $\mathcal C(UA,-)$ and $\mathcal C(UB,-)$ preserve colimits of $\alpha$-chains;
\item  for each $j \in \mathcal J$, the domain and codomain of $Uj$ are $\alpha$-small; or
\item  for each $j \in \mathcal J$, the domain and codomain of $Uj$ are $(\alpha,M)$-small with respect to some proper co-well-powered factorisation system $(E,M)$ on $\mathcal C$,
\end{enumerate}
then the awfs cofibrantly generated by $\mathcal J$ exists.
\end{thm}

\begin{proof}
Combining Theorem~\ref{thm:doublecof} and Proposition \ref{lifting_alg}, this is the case if and only if $U \colon \Tealg \to \mathcal C^{\2b}$ has a left adjoint.  Assuming (1), by Proposition ~\ref{prop:T_finitary}, $T$ then preserves colimits of $\alpha$-chains whilst, assuming (2), the same result ensures that $T$ preserves colimits of $(\alpha,M)$-chains.  The claimed result now holds by the transfinite construction of free algebras on pointed endofunctors described in Theorem~\ref{thm:chains}.
\end{proof}

\section{Cofibrant generation by a small double category of morphisms}\label{sect:double}

This section contains our main result, Theorem \ref{thm:main}, about awfs cofibrantly generated by small double categories.  It improves that of Bourke and Garner \cite{Bourke2016Accessible} by removing the local presentability assumption and by giving an explicit construction via a small object argument.  We compare our construction with that of Garner in Remark \ref{rem:explicit} before establishing a constructive version in Theorem \ref{thm:constructivemain}.

%The main result, Theorem \ref{thm:main}, improves that of Bourke and Garner \cite{BG16} by removing the local presentability assumptions and giving an explicit description using (special) algebraic chains.  We then discuss its constructive aspects in Theorem \ref{thm:constructivemain}.  Our approach to proving these results will be to encode vertical lifting properties using \emph{special} algebras before describing the construction of free special algebras.

\subsection{Encoding double-categorical lifting properties using special algebras}\label{sect:encode}

Let $\mathcal{C}$ be a cocomplete locally small category and consider a double functor $U : \J \to \Sq(\mathcal{C})$ with $\J$ small.  % and the object $Uj$ is $\lambda$-small for every object $j \in \mathcal{J}_0$. Our aim is to show that the forgetful functor $(\mathbb{J}^\dpf)_1 \to \mathcal{C}^{\2b}$ has a left adjoint. We use a similar strategy as for the previous case, but we need an appropriate way of translating the vertical condition in terms of algebras. 
Now let  $\mathcal{J}_1$ be the category of vertical arrows and squares in $\J$ and let $\mathcal{J}_2 = \mathcal{J}_1 \times_{
\mathcal{J}_0} \mathcal{J}_1$ be the category of composable pairs of vertical arrows, and composable squares between them. This comes with a composition functor $m : \mathcal{J}_2 \to \mathcal{J}_1$, which gives the following commutative triangle
\[
\xymatrix{
\mathcal{J}_2 \ar[rd]_{U_2} \ar[rr]^m & & \mathcal{J}_1 \ar[ld]^{U_1} \\
& \mathcal{C}^2 & 
}
\]
with $U_2 = U_1 m$. Now applying Proposition \ref{lifting_alg} twice, we obtain a pair of pointed endofunctors $(T_1,
\eta_1)$ and $(T_2, \eta_2)$ such that $\mathcal{J}_1^\pitchfork \cong \Toalg$ and $\mathcal{J}_2^\pitchfork \cong \Ttwalg$. 

Following Example \ref{1step}, $(\eta_1)_f = (K_1f,1) \colon f \to T_1 f$ comes equipped with a universal $1$-step lifting operation against $\mathcal{J}_1$ denoted by $(j,\sigma) \mapsto \theta_j(\sigma)$.  Likewise, $(\eta_2)_f = (K_2f,1) \colon f \to T_2 f$ has a universal $1$-step lifting operation against $\mathcal{J}_2$ denoted by $((i,j),\sigma) \mapsto \theta_{(i,j)}(\sigma)$, where now $(i,j)$ is a composable pair of vertical morphisms in $\J$.

Observe furthermore that $(\eta_1)_f\colon f \to T_1 f$ has a one-step lifting operation against $\mathcal{J}_2$, which acts first by composing the two vertical arrows and then lifting, as depicted below

% https://q.uiver.app/#q=WzAsNyxbMCwwLCJcXGJ1bGxldCJdLFsxLDAsIlxcYnVsbGV0Il0sWzAsMSwiXFxidWxsZXQiXSxbMCwyLCJcXGJ1bGxldCJdLFsxLDIsIlxcYnVsbGV0Il0sWzIsMiwiXFxidWxsZXQiXSxbMiwwLCJcXGJ1bGxldCJdLFswLDEsIlxccmhvXzAiXSxbMCwyLCJVaSIsMl0sWzIsMywiVWoiLDJdLFszLDQsIlxccmhvXzEiLDJdLFs0LDUsIjEiLDJdLFsxLDYsIktfMWYiXSxbNiw1LCJUXzFmIiwxXSxbMSw0LCJmIiwwLHsibGFiZWxfcG9zaXRpb24iOjcwfV0sWzMsNiwiXFx0aGV0YV97aiBcXGNkb3QgaX0oXFxzaWdtYSkiLDAseyJsYWJlbF9wb3NpdGlvbiI6NDAsInN0eWxlIjp7ImJvZHkiOnsibmFtZSI6ImRhc2hlZCJ9fX1dXQ==
\[\begin{tikzcd}
	\bullet & \bullet & \bullet \\
	\bullet \\
	\bullet & \bullet & \bullet
	\arrow["{\sigma_0}", from=1-1, to=1-2]
	\arrow["Ui"', from=1-1, to=2-1]
	\arrow["{K_1f}", from=1-2, to=1-3]
	\arrow["f"{pos=0.7}, from=1-2, to=3-2]
	\arrow["{T_1f}"{description}, from=1-3, to=3-3]
	\arrow["Uj"', from=2-1, to=3-1]
	\arrow["{\theta_{j \cdot i}(\sigma)}"{pos=0.4}, dashed, from=3-1, to=1-3]
	\arrow["{\sigma_1}"', from=3-1, to=3-2]
	\arrow["1"', from=3-2, to=3-3]
\end{tikzcd}\]

Therefore by Theorem \ref{univ_prop_T} we get a unique morphism $\gamma_f : T_2 f \to T_1 f$ such that $\gamma_f \circ (\eta_2)_f  = (\eta_1)_f$ 
%
%\begin{equation} \label{def_gamma2}
%\xymatrix{
%X \ar@/^1.5pc/[rr]^{K_1f} \ar[d]_f \ar[r]_{K_2f} & S_2f \ar[d]^{T_2f} \ar@{-->}[r]_{(\gamma_f)_0} & S_1 f \ar[d]^{T_1f} \\
%Y \ar[r]_{1_Y} & Y \ar[r]_{1_Y} & Y 
%}
%\end{equation}
and which commutes with the lifting operation, i.e.
\begin{equation} \label{def_gamma}
\theta_{j \cdot i}(\sigma) = (\gamma_f)_0 \circ \theta_{(i,j)}(\sigma)
\end{equation}
for all composable pairs $(i,j)$ in $\mathcal{J}_2$ and $\sigma : U(j \cdot i) \to f$.  It follows from the uniqueness of the $\gamma_f$ that they form the components of a natural transformation $\gamma : T_2 \Rightarrow T_1$. 

The composite square $(\eta_1)_{T_1 f} \circ (\eta_1)_f \colon f \to T_1 f \to T_1T_1f$ also comes equipped with a one-step lifting structure for $f$ against $\mathcal{J}_2$, namely by first lifting against $Ui$ and then against $Uj$ as depicted below:% https://q.uiver.app/#q=WzAsOSxbMCwwLCJcXGJ1bGxldCJdLFsxLDAsIlxcYnVsbGV0Il0sWzAsMSwiXFxidWxsZXQiXSxbMCwyLCJcXGJ1bGxldCJdLFsxLDIsIlxcYnVsbGV0Il0sWzIsMiwiXFxidWxsZXQiXSxbMiwwLCJcXGJ1bGxldCJdLFszLDAsIlxcYnVsbGV0Il0sWzMsMiwiXFxidWxsZXQiXSxbMCwxLCJcXHJob18wIl0sWzAsMiwiVWkiLDJdLFsyLDMsIlVqIiwyXSxbMyw0LCJcXHJob18xIiwyXSxbNCw1LCIxIiwyXSxbMSw2LCJLXzFmIl0sWzYsNywiS18xVF8xIGYiXSxbNSw4LCIxIiwyXSxbNyw4LCJUXzEgVF8xIGYiXSxbNiw1LCJUXzFmIiwxXSxbMSw0LCJmIl0sWzIsNiwiXFx0aGV0YV9pIiwwLHsibGFiZWxfcG9zaXRpb24iOjMwLCJzdHlsZSI6eyJib2R5Ijp7Im5hbWUiOiJkYXNoZWQifX19XSxbMyw3LCJcXHRoZXRhX2oiLDIseyJsYWJlbF9wb3NpdGlvbiI6ODAsInN0eWxlIjp7ImJvZHkiOnsibmFtZSI6ImRhc2hlZCJ9fX1dXQ==
\[\begin{tikzcd}
	\bullet & \bullet & \bullet & \bullet \\
	\bullet \\
	\bullet & \bullet & \bullet & \bullet
	\arrow["{\sigma_0}", from=1-1, to=1-2]
	\arrow["Ui"', from=1-1, to=2-1]
	\arrow["{K_1f}", from=1-2, to=1-3]
	\arrow["f", from=1-2, to=3-2]
	\arrow["{K_1T_1 f}", from=1-3, to=1-4]
	\arrow["{T_1f}"{description}, from=1-3, to=3-3]
	\arrow["{T_1 T_1 f}", from=1-4, to=3-4]
	\arrow["{\theta_i}"{pos=0.3}, dashed, from=2-1, to=1-3]
	\arrow["Uj"', from=2-1, to=3-1]
	\arrow["{\theta_j}"'{pos=0.8}, dashed, from=3-1, to=1-4]
	\arrow["{\sigma_1}"', from=3-1, to=3-2]
	\arrow["1"', from=3-2, to=3-3]
	\arrow["1"', from=3-3, to=3-4]
\end{tikzcd}\]
where we have omitted certain labels for readability --- in symbols, the lifting operation sends $((i,j),\sigma) \mapsto  \theta_{j}(\theta_{i}(\sigma_0,\sigma_1 \circ Uj), \sigma_1)$.

%\[
%\xymatrix{
%\ar[d]_{Ui} \ar[r] & \ar[dd]^f \ar[r] & \ar[dd]^{T_1f} \ar[r] & \ar[dd]^{T_1T_1f} \\
%\ar@{-->}[urr] \ar[d]_{Uj} & & & \\
%\ar@{-->}[uurrr] \ar[r] & \ar[r] & \ar[r] & 
%}
%\]

%One can verify that this is a one-step lifting structure for the square $(K_1T_1f \circ K_1f,1) : f \to T_1T_1f$. 

Therefore, by the universal property of $T_2f$ we get a unique morphism $\lambda_f : T_2f \to T_1T_1f$ such that $\lambda_f \circ (\eta_2)_f  = (K_1T_1f \circ K_1f,1)$ 
%
%making the following diagram commute
%\begin{equation} \label{def_lambda2}
%\xymatrix{
%X \ar@/^1.5pc/[rr]^{K_1T_1f \circ K_1f} \ar[d]_f \ar[r]_{K_2f} & S_2 f \ar[d]^{T_2f} \ar@{-->}[r]_{(\lambda_f)_0} & S_1T_1 f \ar[d]^{T_1T_1f} \\
%Y \ar[r]_{1} & Y \ar[r]_{1} & Y, 
%}
%\end{equation}
and which commutes with the lifting operations: 
\begin{equation} \label{def_lambda}
\theta_j(\theta_i(\sigma_0,\sigma_1 \circ Uj), \sigma_1) = (\lambda_f)_0 \circ \theta_{(i,j)}(\sigma). 
\end{equation}
Again, these form the components of a natural transformation $\lambda : T_2 \Rightarrow T_1 T_1$. 

Now let $\beta : T_1f \to f$ be a $T_1$-algebra, and consider the following commutative diagram 
\begin{equation} \label{special_algebra}
\xymatrix{
T_2 f \ar[d]_{\lambda_f} \ar[rr]^{\gamma_f} & & T_1f \ar[d]^{\beta} \\
T_1T_1f \ar[r]_{T_1\beta} & T_1f \ar[r]_{\beta} & f 
}
\end{equation}
where $\gamma$ and $\lambda$ are as defined above. In the next proposition we prove that $(\mathbb{J}^\dpf)_1$ is isomorphic to the category of $T_1$-algebras that satisfy condition $\eqref{special_algebra}$. 

\begin{prop} \label{lifting_alg_double}
The isomorphism $\mathcal{J}_1^\pitchfork \cong \Toalg$ over $\mathcal{C}^{\2b}$ restricts to an isomorphism between $(\mathbb{J}^\dpf)_1$ and the full subcategory of $\Toalg$ containing those $T_1$-algebras that satisfy condition $\eqref{special_algebra}$.
%Let $\mathcal{C}$ be a cocomplete locally small category, and $U : \J \to \Sq(\mathcal{C})$ a double functor with $\J$ small. Then $(\mathbb{J}^\dpf)_1$ is isomorphic to the category of $T_1$-algebras that satisfy \eqref{special_algebra} over $\mathcal{C}^{\2b}$.  
\end{prop}
\begin{proof}
By Proposition \ref{lifting_alg} we  have an isomorphism $\mathcal{J}_1^\pitchfork \cong \Toalg$ over $\mathcal{C}^{\2b}$. But $(\mathbb{J}^\dpf)_1$ is the full subcategory of $\mathcal{J}_1^\pitchfork$ of lifting structures that additionally satisfy the vertical condition. Hence it suffices to show that $(f,\phi) \in \mathcal{J}_1^\pitchfork$ satisfies the vertical condition if and only if the corresponding $T_1$-algebra $\beta \colon T_1 f \to f$, as described in Remark \ref{rem:beta}, satisfies \eqref{special_algebra}.

Now observe that the two paths $T_2f \rightrightarrows f$ of \eqref{special_algebra} precompose with the unit $(\eta_2)_f$ to give the identity of $f$.  As such, they are $T_2$-algebra structures.  Therefore, these two paths will be equal just when the associated $\mathcal J_2$-lifting structures coincide.

To calculate the associated $\mathcal J_2$-lifting structures consider a vertically composable pair $(i,j)$ and morphism $\sigma: U(j \cdot i) \to f$.  The first lifting structure has component

\begin{align*}
        \beta_0 \circ (\gamma_f)_0 \circ \theta_{(i,j)}(\sigma) {=} \beta_0 \circ \theta_{j \cdot i}(\sigma) = \phi_{j \cdot i}(\sigma), 
    \end{align*}
    where the first equality holds by definition of $\gamma$ as in \eqref{def_gamma} and the second  by definition of $\beta$ as in Remark \ref{rem:beta}.  In other words, it uses the $\mathcal J_1$-lifting structure on $f$ to lift against the composite.
    
The second lifting structure has component
    \begin{align*}
        \beta_0 \circ (T_1 \beta)_0 \circ (\lambda_f)_0 \circ \theta_{(i,j)}(\sigma) & =  & \textrm{(by Equation }\eqref{def_lambda}) &\\
        \beta_0 \circ (T_1 \beta)_0 \circ \theta_{j}(\theta_{i}(\sigma_0, \sigma_1 \circ Uj), \sigma_1)&=  & \textrm{(by Equation }\eqref{morph_prop2}) &\\
        \beta_0 \circ \theta_{j}(\beta_0 \circ \theta_{i}(\sigma_0, \sigma_1 \circ Uj), \sigma_1) &= & \textrm{(by definition of } \beta \textrm{ as in Remark }  \ref{rem:beta}) &\\
        \beta_0 \circ \theta_{j}(\phi_i(\sigma_0, \sigma_1 \circ Uj), \sigma_1)  &= & \textrm{(by definition of } \beta \textrm{ as in Remark }  \ref{rem:beta}) &\\
       \phi_j(\phi_i(\sigma_0, \sigma_1 \circ Uj), \sigma_1)
    \end{align*}
%     \begin{align*}
%        \beta_0 \circ (T_1 \beta)_0 \circ (\lambda_f)_0 \circ \theta_{(i,j)}(\sigma)  & = \beta_0 \circ (T_1 \beta)_0 \circ \theta_{j}(\theta_{i}(\sigma_0, \sigma_1 \circ Uj), \sigma_1) & \textrm{(by Equation }\eqref{def_lambda}) \\
%        & = \beta_0 \circ \theta_{j}(\beta_0 \circ \theta_{i}(\sigma_0, \sigma_1 \circ Uj), \sigma_1)  & \textrm{(by Equation }\eqref{morph_prop2}) \\
%        & = \beta_0 \circ \theta_{j}(\phi_i(\sigma_0, \sigma_1 \circ Uj), \sigma_1)  & \textrm{(by definition of } \beta \textrm{ as in Remark }  \ref{rem:beta}) \\
%       & = \phi_j(\phi_i(\sigma_0, \sigma_1 \circ Uj), \sigma_1) & \textrm{(by definition of } \beta \textrm{ as in Remark }  \ref{rem:beta})
%    \end{align*}
%        \begin{align*}
%        \beta \circ T_1 \beta \circ \lambda_f \circ \theta_{(i,j)}(\sigma)  & = \beta \circ T_1\beta \circ \theta_{j,T_1f}(\theta_{i,f}(\sigma_0, \sigma_1 \circ Uj), \sigma_1) & \textrm{(by Equation }\eqref{def_lambda}) \\
%        & = \beta \circ \theta_{j,f}(\beta \circ \theta_{i,f}(\sigma_0, \sigma_1 \circ Uj), \sigma_1)  & \textrm{(by Equation }\eqref{morph_prop2}) \\
%        & = \beta \circ \theta_{j,f}(\phi_i(\sigma_0, \sigma_1 \circ Uj), \sigma_1)  & \textrm{(by definition of } \beta \textrm{ as in Remark }  \ref{rem:beta}) \\
%       & = \phi_j(\phi_i(\sigma_0, \sigma_1 \circ Uj), \sigma_1) & \textrm{(by definition of } \beta \textrm{ as in Remark }  \ref{rem:beta})
%    \end{align*}
which is to say, it first lifts against $i$ and then against $j$.  Therefore, to say that the two $\mathcal J_2$-lifting structures coincide is precisely to say that the vertical compatibility condition holds, as required.

\end{proof}
\noindent Let us call a $T_1$-algebra special if diagram \eqref{special_algebra} commutes, and denote the full subcategory of special $T_1$-algebras by $\Toalg^s$. Proposition \ref{lifting_alg_double} shows that we have an isomorphism $(\mathbb{J}^\dpf)_1 \cong \Toalg^s$ over $\mathcal{C}^{\2b}$.  Hence to show that the forgetful functor $(\mathbb{J}^\dpf)_1 \to \mathcal{C}^{\2b}$ has a left adjoint is equally to show that the forgetful functor $\Toalg^s \to \mathcal{C}^{\2b}$ has a left adjoint, i.e. that the free special $T_1$-algebra exists. %This is proved in the next section. 

\subsection{Existence of the free special algebra}\label{sect:special}

%By the preceding, the problem of constructing an awfs on cocomplete category, cofibrantly generated by small double categories, is reduced to the problem of constructing free special $T_1$-algebras.  The present section is devoted to solving this problem.  

In what follows, we work in the more general setting of a pair of pointed endofunctors $(T_1,\eta_1)$ and $(T_2,\eta_2)$ on a cocomplete category $\mathcal{C}$ together with natural transformations $\gamma : T_2 \Rightarrow T_1$ and $\lambda : T_2 \Rightarrow T_1T_1$. We define $\Toalg^s$ as the full subcategory of `special' $T_1$-algebras $\beta : T_1X \to X$, i.e. those for which the following diagram commutes
   \begin{equation} \label{special_algebra_gen}
    \xymatrix{
    T_2 X \ar[d]_{\lambda_X} \ar[rr]^{\gamma_X} & & T_1X \ar[d]^{\beta} \\
    T_1T_1X \ar[r]_{T_1\beta} & T_1X \ar[r]_{\beta} & X. 
    }
    \end{equation}
Our aim to show that the free special $T_1$-algebra exists. The proof is an adaption of the proof we gave in Section \ref{section:free_alg} for the existence of free algebras on pointed endofunctors. 

Firstly, we define the category of special algebraic chains. 

\begin{defi} \label{special_alg_chain}
    A \textit{special} \emph{algebraic chain} $(X,x)$ is an algebraic chain (for $T_1$) with the additional condition that for all $n$ the diagram 
        \begin{equation} \label{diag:sp_alg_ch}
        \xymatrix{
        T_2 X_n \ar[d]_{\lambda_{X_n}} \ar[r]^{\gamma_{X_n}} & T_1X_n \ar[r]^{x_n} & X_{n+1} \ar[d]^{j_{n+1}^{n+2}} \\
        T_1T_1X_n \ar[r]_{T_1x_n} & T_1X_{n+1} \ar[r]_{x_{n+1}} & X_{n+2} 
        }
        \end{equation}
    commutes. This is a full subcategory of $\Toalg_\infty$ which we denote by $\Tsalg$. 
\end{defi}

Note that the constant functor $\Delta$ restricts to $\Delta: \Toalg^s \to \Tsalg$. Indeed, given a special algebra $(X,\beta)$, diagram \eqref{diag:sp_alg_ch} for the constant algebraic chain $\Delta(X,\beta)$ simply collapses to diagram $\eqref{special_algebra_gen}$ for every $n$. Thus, we have a diagram 
\begin{equation*}
\xymatrix{
\Toalg^s \ar[dr]_{U} \ar[r]^{\Delta} & \Tsalg \ar[d]^{V} \\
& \mathcal C}
\end{equation*}
and we can again break down the problem of constructing the free special algebra into two parts: (1) constructing the free special algebraic chain $X^{\star}$ on  $X$ and (2) establishing when it admits a reflection along $\Delta$.  

Again, the interesting part is to draw out the formula for the free special algebraic chain.  Recall from the argument above Proposition \ref{prop:free} that a chain together with maps $x_n \colon T_1 X_n \to X_{n+1}$ satisfying the unit equation $x_{n} \circ \eta_{X_{n}} = j^{n+1}_n$ is an algebraic chain just when the two diagrams on the line below

\begin{equation*}
\xymatrix{
T_1 X_{n} \ar@<.6ex>[rr]^{T_1 x_{n} \circ T_1\eta_{X_{n}}}\ar@<-.6ex>[rr]_{T _1 x_{n} \circ \eta_{T_1 X_{n}}} && T_1 X_{n+1} \ar[r]^{x_{n+1}} & X_{n+2}
}
\hspace{1cm}
\xymatrix{
\colim_{n<m}T_1X_{n} \ar@<0.5ex>[rr]^(.6){\langle T_1{j_n^m} \rangle} \ar@<-0.5ex>[rr]_(.6){\langle \eta_{X_{m}} \circ j_{n+1}^{m} \circ x_{n} \rangle} && T_1X_{m} \ar[r]^{x_{m}} & X_{m+1}
}
\end{equation*}
are forks, where in the second case $m$ is a limit ordinal.   Assuming these equations hold, the equation \eqref{diag:sp_alg_ch} for a special algebraic chain becomes the assertion that the following diagram
\begin{equation*}
\xymatrix{
T_2 X_{n} \ar@<.6ex>[rr]^{T_1 x_{n} \circ \lambda_{X_{n}}}\ar@<-.6ex>[rr]_{T_1 (j_n^{n+1}) \circ \gamma_{X_{n}}} && T_1 X_{n+1} \ar[r]^{x_{n+1}} & X_{n+2}
}
\end{equation*}
is a fork, using that $x_{n+1} \circ T_1 j_n^{n+1} = j_{n+1}^{n+2} \circ x_n$.  Therefore, the only change overall is that $X_{n+2}$ must coequalise two forks, which leads to the formula in the following proposition, describing the value of the free special algebraic chain $X^\star$ at $n+2$ as a joint coequaliser.

\begin{prop} \label{prop:free2} 
    If $\mathcal{C}$ is cocomplete, then the forgetful functor $V : \Tsalg \to \mathcal{C}$ has a left adjoint which sends an object $X \in \mathcal{C}$ to the algebraic chain $X^\star$ defined as follows: 
    \begin{itemize}
        \item $X^\star_0 = X, X^\star_1 = T_1X, j_0^1 = \eta_X : X \to T_1X$ and $x_0 = 1 : T_1X \to T_1X$. 
        \item For all $n$, $X^\star_{n+2}$ and $x_{n+1}$ are defined as the joint coequaliser: 
        \[\xymatrix{
 T_1X^\star_n   \ar@<1ex>[drr]^{T_1x_n \circ T_1\eta_{X^\star_n}}   \ar@<-1ex>[drr]_{T_1x_n \circ \eta_{T_1X^\star_n}} \\
 & &  T_1X^\star_{n+1} \ar[r]^{x_{n+1}} & X^\star_{n+2} \\
  T_2X^\star_n \ar@<1ex>[urr]^{T_1x_n \circ \lambda_{X^\star_n}} \ar@<-1ex>[urr]_{T_1(j_n^{n+1}) \circ \gamma_{X^\star_n}}
  }
  \]        and $j_{n+1}^{n+2} = x_{n+1} \circ \eta_{X^\star_{n+1}}$. 
        \item At a limit ordinal $\alpha$,
\begin{itemize} 
\item $X^{\star}_{\alpha} = \colim_{n < \alpha} X^{\star}_{n}$ with the connecting maps $j_n^\alpha$ the colimit inclusions.  
\item $X^{\star}_{\alpha + 1}$ is the coequaliser
\begin{equation*}
\xymatrix{
\colim_{n<m}TX^{\star}_{n} \ar@<0.5ex>[rr]^(.6){\langle T{j_n^\alpha} \rangle} \ar@<-0.5ex>[rr]_(.6){\langle \eta_{X^{\star}_{\alpha}} \circ j_{\alpha +1}^{\alpha} \circ x_{n} \rangle} && TX^{\star}_{\alpha} \ar[r]^{x_{\alpha}} & X^{\star}_{\alpha+1}
}
\end{equation*}
with $j_{\alpha}^{\alpha+1}=x_{\alpha} \circ \eta_{X^{\star}_{\alpha}}$.
\end{itemize}
    \end{itemize}
\end{prop}
\begin{proof}
The proof is identical in form to that of Proposition \ref{prop:free} with the exception that in extending from $n+1$ to $n+2$ we use the universal property of $X^\star_{n+2}$ as a joint coequaliser rather than a coequaliser.
 \end{proof}

Turning to part (2), we have seen in section \ref{section:free_alg} that if an algebraic chain $(X,x)$ stabilises at $n$ then $X_{n}$ equipped with the $T_1$-algebra structure
\begin{equation}\label{eq:structure2}
(j_n^{n+1})^{-1} \circ x_n:T_1X_n \to X_{n+1} \cong X_n
\end{equation} is a reflection of $(X,x)$ along $\Delta$. In particular, in this case $X_{n}$, with structure map as in \eqref{eq:structure2}, is the free $T_1$-algebra on $X$. Moreover, if $m$ is a limit ordinal and $T_1$ preserves the colimit $X^{\star}_m = \colim_{n < m} X^{\star}_n$, then $X^{\star}$ stabilises. Thus, it suffices to show that if $(X,x)$ is a special algebraic chain, then $X_n$ equipped with the structure map \eqref{eq:structure2} is in fact a special algebra.

\begin{prop}\label{prop:stablefree2}
If the special algebraic chain $(X,x)$ stabilises at an ordinal $n$, then the $T_1$-algebra 
\begin{equation}\label{eq:structure3}
\beta = (j_n^{n+1})^{-1} \circ x_n:TX_n \to X_{n+1} \cong X_n
\end{equation}
of Proposition \ref{prop:stablefree} is special, and is a reflection along $\Delta$. 
\end{prop}
\begin{proof}
By Proposition \ref{prop:stablefree}, it suffices to prove that the $T_1$-algebra structure is special.  For arbitrary $n$, we compute: 
    \begin{align*}
        \beta \circ T_1\beta \circ \lambda_{X_{n}}  & = & \textrm{(by definition of } \beta) & \\
      (j_{n}^{n+1})^{-1} \circ x_{n} \circ T_1(j_{n}^{n+1})^{-1} \circ T_1x_{n} \circ \lambda_{X_{n}}  & = & \textrm{(by definition of an algebraic chain})& \\
        (j_{n}^{n+1})^{-1} \circ (j_{n+1}^{n+2})^{-1} \circ x_{n+1} \circ T_1x_{n} \circ \lambda_{X_{n}} & = & \textrm{(on composing inverses})& \\
        (j_{n}^{n+2})^{-1} \circ x_{n+1} \circ T_1x_{n} \circ \lambda_{X_{n}} & = & \textrm{(by Equation }\eqref{diag:sp_alg_ch} \textrm{ for a special algebraic chain})& \\
         (j_{n}^{n+2})^{-1} \circ j_{n+1}^{n+2} \circ x_n \circ \gamma_{X_n}  & = & \textrm{(by functoriality of chains})& \\
         (j_{n}^{n+1})^{-1} \circ x_n \circ \gamma_{X_n} & =  & \textrm{(by definition of } \beta)& \\
         \beta \circ \gamma_{X_{n}} 
    \end{align*}
\end{proof}

\begin{thm}\label{thm:chains2}
Let $(T_1,\eta_1)$ and $(T_2,\eta_2)$ be pointed endofunctors on a cocomplete category $\mathcal{C}$ together with natural transformations $\gamma : T_2 \Rightarrow T_1$ and $\lambda : T_2 \Rightarrow T_1T_1$.  If either
\begin{enumerate}
\item[(a)] $T_1$ and $T_2$ preserve colimits of $\alpha$-chains for some limit ordinal $\alpha$, or
\item[(b)] $\mathcal C$ is equipped with a co-well-powered proper factorisation system $(E,M)$ such that $T_1$ and $T_2$ preserves colimits of $(\alpha,M)$-chains for some limit ordinal $\alpha$.
\end{enumerate}
Then (1) each special algebraic chain $X^{\star}$ stabilises and (2) its point of stabilisation, with algebra structure as in \eqref{eq:structure3}, is the free special $T$-algebra on $X$.
\end{thm}
\begin{proof}
We defer the proof that $X^{\star}$ stabilises to the appendix.  Since the triangle
\begin{equation*}
\xymatrix{
\Toalg^s \ar[dr]_{U} \ar[r]^{\Delta} & \Tsalg \ar[d]^{V} \\
& \mathcal C}
\end{equation*}
commutes, the free special $T$-algebra is given by the reflection of the free algebraic chain $X^{\star}$ of Proposition \ref{prop:free2} along $\Delta$.  By Proposition \ref{prop:stablefree2}, this reflection is the point of stabilisation of $X^{\star}$. \end{proof}

\subsection{The small object argument for double-categorical cofibrant generation}

We are now ready to give the main result on cofibrant generation by double categories of morphisms.  This improves Bourke and Garner's result, Proposition 23 of \cite{Bourke2016Accessible}, by removing the local presentability assumption and by giving an explicit description of the construction.

\begin{thm}\label{thm:main}
Let $\mathcal{C}$ be a cocomplete locally small category, $\mathbb J$ a small double category and consider $U : \mathbb J \to \Sq(\mathcal{C})$.  If there exists a limit ordinal $\alpha$ such that either
\begin{enumerate}
%\item there exists a limit ordinal $\alpha$ such that for each $j \colon A \to B \in J$, $\mathcal C(UA,-)$ and $\mathcal C(UB,-)$ preserve colimits of $\alpha$-chains;
\item  for each object $A$ of $\mathbb J$, $UA$ is $\alpha$-small in $\mathcal C$; or 
%for each $j \in \mathcal J_1$, the domain and codomain of $Uj$ are $\alpha$-small; or
\item   for each object $A$ of $\mathbb J$, $UA$ is
%for each $j \in \mathcal J_1$, the domain and codomain of $Uj$ are
 $(\alpha,M)$-small with respect to some proper co-well-powered factorisation system $(E,M)$ on $\mathcal C$,
\end{enumerate}
then the awfs cofibrantly generated by $\mathbb J$ exists.
\end{thm}
\begin{proof}
Combining Theorem~\ref{thm:doublecof} and Proposition \ref{lifting_alg_double}, this is the case if and only if the forgetful functor $U: \Toalg^s \to \mathcal{C}^{\2b}$ has a left adjoint.  Assuming (1), by Proposition ~\ref{prop:T_finitary}, $T_1$ and $T_2$ then preserve colimits of $\alpha$-chains; assuming (2), the same result ensures that both preserve colimits of $(\alpha,M)$-chains.  The claimed result now holds by the transfinite construction of free special algebras on pointed endofunctors described in Theorem~\ref{thm:chains2}.
\end{proof}

\begin{rem}\label{rem:explicit}
Let us now spell out the explicit small object argument of Theorem \ref{thm:main} and compare it with Garner's small object argument described in the introduction.  Our construction is an instance of the free special algebraic chain described in Proposition \ref{prop:free2} whilst Garner's construction is an instance of the free algebraic chain construction of Proposition \ref{prop:free}.  

Both begin in the same way.  Given a morphism $f \in \mathcal C$, we form $\eta_f \colon f \to T_1 f$ where $f^{\bullet}_1:= T_1 f$ is universally equipped (in the sense of Section \ref{sect:categorical}) with natural fillers for lifting problems against morphisms of $\mathcal J_1$ against $f$.  

In Garner's small object argument against $\mathcal J_1$, in the next step one forms the coequaliser
\[
        \xymatrix{
        Tf \ar@<1ex>[rr]^{T\eta_{f}} 
        \ar@<-1ex>[rr]_{\eta_{Tf}} & & T^{2}f \ar[r]^{} &  f^{\bullet}_2 \\
        }
\]
 which has the effect of adding fillers for lifting problems against $Tf$, and then quotienting out to identify those fillers for lifting problems whose solutions had already been added in the first stage.  

 In the double-categorical version, we do not just form a coequaliser at the second stage but a joint coequaliser as below.  
\[
\xymatrix{
 T_1f   \ar@<1ex>[drr]^{T\eta_{f}}   \ar@<-1ex>[drr]_{\eta_{Tf}} \\
 & &  (T_1)^2 f \ar[r]^{} & f^\star_{2} \\
  T_2f \ar@<1ex>[urr]^{\lambda_{f}} \ar@<-1ex>[urr]_{T_1\eta_f \circ \gamma_{f}}
  }
  \]

Here, the upper fork is the same as above whilst the lower fork of the joint coequaliser enforces the vertical compatibility condition for lifting problems against composable pairs of vertical morphisms (i.e. morphisms of $\mathcal J_2$) which are encoded by $T_2$.

At the later stages, we keep forming joint coequalisers, one fork of which quotients out redundant fillers  whilst the others enforces the vertical compatibility condition. (The horizontal compatibility condition is already encoded in the construction of the functor $T_1$.)
\end{rem}

\subsection{Constructive aspects}

From a constructive point of view most of our arguments up to this point are valid. There are two exceptions to this:
\begin{enumerate}
    \item[(1)] For certain ordinals it may be undecidable whether elements are zero, a successor or a limit. This means that case distinctions as in Propositions \ref{prop:free} and \ref{prop:free2} will not always be possible.
    \item[(2)] We do not see how to give a constructive proof of the clever lemma from Koubek and Reiterman \cite{Koubek1979Categorical} that we use in the appendix to reduce condition (b) to condition (a) in Theorems \ref{thm:chains} and \ref{thm:chains2}.
\end{enumerate}
We do not see how to overcome the second problem; however, the first problem does not occur for relatively small ordinals like $\omega$ and $\omega + \omega$. For that reason our methods still yield a constructive proof of the following result.

\begin{thm}\label{thm:constructivemain} {\rm (Constructive)}
Let $\mathcal{C}$ be a cocomplete locally small category, $\mathbb J$ a small double category and consider $U : \mathbb J \to \Sq(\mathcal{C})$.  
 %Let $\mathcal{C}$ be a cocomplete locally small category and consider a small double category $U : \mathbb J \to \Sq(\mathcal{C})$ over $\Sq(\mathcal{C})$.  
If for each object $A$ of $\mathbb J$, $UA$ is $\omega$-small in $\mathcal C$, then the awfs cofibrantly generated by $\mathbb J$ exists.
\end{thm}
\begin{proof}
    The result that we need in this case is that for functors $T_1$ and $T_2$ that preserve $\omega$-filtered limits (``are finitary'') the special free algebra exists. If we restrict our attention to chains of length $\omega + \omega$, then the formula for the free algebraic chain given in Proposition \ref{prop:free2} is constructively acceptable, as for ordinals less than $\omega + \omega$, the order is decidable as is the question whether an ordinal is zero, a successor or the limit $\omega$. Since both $T_1$ and $T_2$ preserve $\omega$-filtered colimits, these chains stabilise at the ordinal $\omega$ at which point they yield the free special algebra. For further details about the finitary case we refer to \cite{Seip2024}. 
\end{proof}

\begin{corr} \label{effKanfibrawfs} {\rm (Constructive)}
    The effective Kan fibrations as introduced in \cite{BergFaber2022} form the right class in an awfs.
\end{corr}
\begin{proof}
    In \cite{BergGeerligs25} it is shown that the effective Kan fibrations from \cite{BergFaber2022} are precisely those maps which have the right lifting property against a small double category $U : \mathbb J \to \Sq(\mathcal{\bf sSets})$ such that each $UA$ is a decidable sieve \cite[Definition 4.5]{BergGeerligs25}. By \cite[Lemma 8.1]{BergFaber2022} such decidable sieves are generated by a finite set of mononorphisms, which implies that they are a finite colimit of representables and hence $\omega$-small. Therefore for each object $A$ of $\mathbb{J}$ the  object $UA$ is $\omega$-small and the result follows from Theorem \ref{thm:constructivemain}.
\end{proof}

%\section{Conclusion}
%
%We have given an explicit construction of the algebraic weak factorisation system generated by a small double category. This made it possible to formulate a constructive version of this result, which we used to show that the effective Kan fibrations from \cite{BergFaber2022} form part of an algebraic weak factorisation system.
%\nocite{*}
\printbibliography

\newpage 

\appendix
\markboth{}{}
\section{Deferred proofs about the stabilisation of chains}

\begin{proof}[Proof of Theorem~\ref{thm:chains}]
\emph{(1) We start by showing that under either assumption (a) or (b), there exists a limit ordinal $\alpha$ such that $T$ preserves the colimit of $(X^{\bullet}_n)_{n < \alpha}$}. 
Indeed, assuming (b), the lemma in Section 8.5 of Koubek and Reiterman \cite{Koubek1979Categorical} proves that if $A$ is any chain, then there exists a limit ordinal $\alpha$ such that $T$ preserves the colimit of $(A_{n})_{n < \alpha}$.  (This is also proved in Proposition 4.2 of \cite{Kelly1980A-unified}.)  Therefore, without loss of generality, we can assume in either case, that there exists a limit ordinal $\alpha$ such that $T$ preserves the colimit of $(X^{\bullet}_n)_{n < \alpha}$.  Assuming this, we proceed to show that $X^{\bullet}$ stabilises at $\alpha$.

\emph{(2) We start by proving that $j_\alpha^{\alpha +1}$ is invertible}. To this end, firstly note that the maps $x_n : TX^\bullet_n \to X^\bullet_{n+1}$ for $n < \alpha$ form a morphism between chains of length $\alpha$ and so induce a unique morphism between their colimits $x'_{\alpha} \colon TX^{\bullet}_\alpha \to X^{\bullet}_\alpha$ such that 
\begin{equation} \label{eqn:a} x'_{\alpha} \circ Tj^\alpha_n = j^\alpha_{n+1} \circ x_n.\end{equation}
From this, the equality \begin{equation} \label{eqn:b} x'_{\alpha} \circ \eta_{X^\bullet_\alpha} = 1_{X^\bullet_\alpha} \end{equation} is easily deduced --- indeed, precomposing with the colimit inclusions we have
\begin{align*} x'_\alpha \circ \eta_{X^\bullet_\alpha} \circ j_n^\alpha  & = x'_\alpha \circ T(j_n^\alpha) \circ \eta_{X^\bullet_n} & \textrm{(by naturality of $\eta$)}\\
    & = j_{n+1}^\alpha \circ x_n \circ \eta_{X^\bullet_n} & \textrm{(by Equation \ref{eqn:a})} \\ & = j_{n+1}^\alpha \circ j_{n}^{n+1} & \textrm{(by definition of an algebraic chain)} \\
    & = j_n^\alpha. & \textrm{(by functoriality of $j$)}
\end{align*}
Similarly, the triangle
\[
    \xymatrix{
    TX^\bullet_\alpha \ar[rd]_{x_\alpha} \ar[r]^{x'_\alpha} & X^\bullet_\alpha \ar@{-->}[d]^{j_\alpha^{\alpha+1}} \\
    & X^\bullet_{\alpha+1} 
    }
 \]
 commutes, since precomposing with the colimit inclusions, we have  
\begin{align*}
    j_\alpha^{\alpha +1} \circ x'_\alpha \circ Tj^\alpha_n & = j_\alpha^{\alpha +1} \circ j^\alpha_{n+1} \circ x_n & \textrm{(by Equation \ref{eqn:a})}\\  
    & = j^{\alpha+1}_{n+1} \circ x_n & \textrm{(by functoriality of $j$)}\\
    & = x_\alpha \circ Tj^\alpha_n. & \textrm{(by definition of an algebraic chain)}
\end{align*}
Since  $x_\alpha$ is defined as a coequaliser map, the commutativity of the above triangle ensures that $j_\alpha^{\alpha+1}$ is invertible if and only if the diagram on the top row below is a coequaliser diagram.
    \[
    \xymatrix{
    \colim_{n < \alpha} TX^\bullet_n \ar@<1ex>[rr]^(.6){\langle Tj_n^\alpha \rangle} \ar@<-1ex>[rr]_(.6){\langle \eta_{X^\bullet_\alpha} \circ j_{n+1}^\alpha \circ x_n \rangle} & & TX^\bullet_{\alpha} \ar[rd]_{h} \ar[r]^{x'_\alpha} & X^\bullet_{\alpha} \ar@{.>}[d]^{\exists! \varphi} \\
    & & & C
    }
    \]
            
To show that it is a fork, we compute for $n < \alpha$: 
    \begin{align*}
        x'_\alpha \circ Tj_n^\alpha \justeq{(\ref{eqn:a})} j_{n+1}^\alpha \circ x_n \justeq{(\ref{eqn:b})} x'_\alpha \circ \eta_{X^\bullet_\alpha} \circ j_{n+1}^\alpha \circ x_n. 
    \end{align*}

To verify the universal property, assume that we have a fork $h : TX^\bullet_\alpha \to C$ as depicted above.  We must construct a morphism $\varphi \colon X^\bullet_{\alpha} \to C$ making the diagram commute.  (Its uniqueness then follows automatically since $x'_{\alpha}$ is a split epi.)
    
    To define $\varphi$, note that we have a cocone $(h \circ Tj_n^\alpha \circ \eta_{X^\bullet_n})_{n \geq 1}$ on $(X^\bullet_n)_{n < \alpha}$ with vertex $C$, as witnessed by the following commutative diagram
    \[
    \xymatrix{
    X^\bullet_k \ar[d]_{j_k^n} \ar[r]^{\eta_{X^\bullet_k}} & TX^\bullet_k \ar[d]_{Tj_k^n} \ar[r]^{Tj_k^\alpha} & TX^\bullet_\alpha \ar[d]_1 \ar[r]^{h} & C \ar[d]^1 \\
    X^\bullet_n \ar[r]_{\eta_{X^\bullet_n}} & TX^\bullet_n \ar[r]_{Tj_n^\alpha} & TX^\bullet_\alpha \ar[r]_h & C.
    }
    \]
    Thus we get a unique map $\varphi : X^\bullet_\alpha \to C$ with the property that for all $n < \alpha$, 
    \begin{equation} \label{eqn:c}
    \varphi \circ j_n^\alpha = h \circ Tj_n^\alpha \circ \eta_{X^\bullet_n}. 
    \end{equation}
    Our goal is to show that $\varphi \circ x'_\alpha = h$ but let us firstly show that \begin{equation} \label{eqn:d} h \circ \eta_{X^\bullet_\alpha} = \varphi. \end{equation}  Indeed, upon precomposition with the colimit inclusions and using that $\eta$ is natural, we have
    \[
    h \circ \eta_{X^\bullet_\alpha} \circ j_n^\alpha = h \circ Tj_n^\alpha \circ \eta_{X^\bullet_n} \justeq{(\ref{eqn:c})} \varphi \circ j_n^\alpha
    \]
   as required. Then to see that $\varphi \circ x'_\alpha = h$ we again precompose with colimit inclusions, computing: 
$$ h \circ Tj_n^\alpha =    h \circ \eta_{X^\bullet_\alpha} \circ j_{n+1}^\alpha \circ x_n \justeq{(\ref{eqn:d})}   \varphi \circ j_{n+1}^\alpha \circ x_n \justeq{(\ref{eqn:a})}  \varphi \circ x'_\alpha \circ Tj_n^\alpha$$
   where in the first equality we use that $h$ is a fork.  
   
\emph{(3) Next, we prove that if $j_{n}^{n+1}$ is invertible, so is $j_{n+1}^{n+2}$}.  Note that the triangle in the diagram below commutes

    \[
    \xymatrix{
    TX^\bullet_n \ar@<1ex>[rr]^{Tx_n \circ T\eta_{X^\bullet_n}} 
    \ar@<-1ex>[rr]_{Tx_n \circ \eta_{TX^\bullet_n}} & & TX^\bullet_{n+1} \ar[rrd]_{x_{n+1}} \ar[rr]^{x_n \circ T(j_n^{n+1})^{-1}} & & X^\bullet_{n+1} \ar@{-->}[d]^{j_{n+1}^{n+2}} \\
    & & & & X^\bullet_{n+2}
    }\]
  since $j_{n+1}^{n+2} \circ x_n = x_{n+1} \circ T(j_n^{n+1})$ by the definition of an algebraic chain.  Since the lower fork is, by definition, a coequaliser, $j_{n+1}^{n+2}$ will be invertible just when the top row above is a coequaliser diagram. 
Firstly we show that it is a fork. We calculate, on the one hand:

\begin{align*} x_n \circ T(j_n^{n+1})^{-1} \circ Tx_n \circ T\eta_{X^\bullet_n} & = x_n \circ T(j_n^{n+1})^{-1} \circ T(j_n^{n+1}) & \textrm{(by definition of algebraic chain)} \\
    & = x_n  & \textrm{(by cancelling inverses)}  \end{align*}
%    \begin{align*}
%        x_n \circ T(j_n^{n+1})^{-1} \circ Tx_n \circ T\eta_{X_n} & = x_n \circ T((j_n^{n+1})^{-1} \circ x_n \circ \eta_{X_n}) \\
%        & = x_n \circ T((j_n^{n+1})^{-1} \circ j_n^{n+1}) \\ 
%        & = x_n \circ T(1_{X_{n}}) = x_n,
%    \end{align*}
whilst on the other:
\begin{align*} x_n \circ T(j_n^{n+1})^{-1} \circ Tx_n \circ \eta_{TX^\bullet_n} & = x_n \circ \eta_{X^\bullet_n} \circ (j_n^{n+1})^{-1} \circ x_n & \textrm{(by naturality of $\eta$)}\\
    & =  j_n^{n+1} \circ (j_n^{n+1})^{-1} \circ x_n & \textrm{(by definition of algebraic chain)}\\
    & = x_n & \textrm{(by cancelling inverses)} \end{align*}
%        & = x_n. 
%    \begin{align*}
%        x_n \circ T(j_n^{n+1})^{-1} \circ Tx_n \circ \eta_{TX_n} & = x_n \circ T((j_n^{n+1})^{-1} \circ x_n) \circ \eta_{TX_n} \\
%        & = x_n \circ \eta_{X_n} \circ (j_n^{n+1})^{-1} \circ x_n \\
%        & = j_n^{n+1} \circ (j_n^{n+1})^{-1} \circ x_n \\
%        & = x_n. 
%    \end{align*}
%    For the second equality we have used the naturality diagram 
%    \[
%    \xymatrix{
%    TX_n \ar[d]_{(j_n^{n+1})^{-1} \circ x_n} \ar[r]^{\eta_{TX_n}} & TTX_n \ar[d]^{T((j_n^{n+1})^{-1} \circ x_n)} \\
%    X_n \ar[r]_{\eta_{X_n}} & TX_n. 
%    }
%    \]

In fact, we have a split coequaliser.
  \[
    \xymatrix{
    TX^\bullet_n \ar@<2ex>[rr]^{Tx_n \circ T\eta_{X^\bullet_n}} 
    \ar@<-1ex>[rr]^{Tx_n \circ \eta_{TX^\bullet_n}} & & TX^\bullet_{n+1}  \ar@<3ex>[ll]^{T(j_n^{n+1})^{-1}}  \ar[rr]^{x_n \circ T(j_n^{n+1})^{-1}} & & X^\bullet_{n+1}   \ar@<2ex>[ll]^{\eta_{TX^\bullet_{n+1}}} 
    }\]
The three required equations are $$x_n \circ T(j_n^{n+1})^{-1} \circ \eta_{TX^\bullet_{n+1}} = x_n \circ \eta_{X^\bullet_{n+1}} \circ (j_n^{n+1})^{-1} = {j_n^{n+1}} \circ (j_n^{n+1})^{-1} = 1_{X^\bullet_{n+1}}, $$

$$\eta_{X^\bullet_{n+1}} \circ x_n \circ T(j_n^{n+1})^{-1}  = Tx_n \circ \eta_{TX^\bullet_{n}}  \circ T(j_n^{n+1})^{-1}$$
and 
$$Tx_n \circ T\eta_{X^\bullet_{n}} \circ T(j_n^{n+1})^{-1} = T(j_n^{n+1}) \circ T(j_n^{n+1})^{-1} = 1_{TX^\bullet_{n+1}},$$
where in each of the steps we have used either the naturality of $\eta$ or the definition of an algebraic chain.

\emph{(4) Finally we prove that $j^\beta_{\alpha}$ is invertible for all $\alpha \leq \beta.$}  We do so by transfinite induction.  The base case is established in (1) above so it remains to prove the inductive step.

Firstly, if $\beta$ is a limit ordinal, then suppose that $j_{\alpha}^{n}$ is invertible for all $\alpha \leq n < \beta$.  Then since $X^{\bullet}_{\beta} = \colim_{n < \beta} X^{\bullet}_{n} = \colim_{\alpha \leq n < \beta} X^{\bullet}_n$, it is the colimit of a chain of isomorphisms, and hence each colimit coprojection $j^{\beta}_{n}$ is invertible, as required.

Secondly, suppose $\beta$ is a successor.  The case in which $\beta = \gamma+2$ follows from (2) above using the inductive hypothesis.  Thus, we are left to consider the case $\beta = \gamma + 1$ for $\gamma$ a limit ordinal, and using the inductive hypothesis, and that $j^\gamma_\alpha$ is invertible, this amounts to showing that $j_{\gamma}^{\gamma+1}$ is invertible.  

Choose then $\alpha \leq n < \gamma$ so that, by assumption, $j^{\gamma}_{n}$ is invertible.  Then  the triangle in the diagram below commutes by the definition of an algebraic chain.  Since the lower fork is a coequaliser, it suffices to show that the upper row is a coequaliser too.

    \[
    \xymatrix{
    \colim_{n < \gamma} TX^\bullet_n \ar@<1ex>[rr]^(.6){\langle Tj_n^\gamma \rangle} \ar@<-1ex>[rr]_(.6){\langle \eta_{X^\bullet_\gamma} \circ j_{n+1}^\gamma \circ x_n \rangle} & & TX^\bullet_{\gamma} \ar[rrrd]_{x_{\gamma}} \ar[r]^(.6){T(j^{\gamma}_{n})^{-1}} & TX^\bullet_{n} \ar[r]^{x_n} & X^\bullet_{n+1} \ar[r]^{j^{\gamma}_{n+1}} & X^\bullet_{\gamma } \ar@{.>}[d]^{ j_{\gamma}^{\gamma+1}} \\
& &  & &    & X^\bullet_{\gamma+1}
    }
    \]
We leave the straightforward verification that it is a fork to the reader.   In fact it is a split fork, with splitting given by the pair $\eta_{X^\bullet_{\gamma }} \colon X^\bullet_{\gamma } \to TX^\bullet_{\gamma }$ and $\iota_n \circ T(j^{\gamma}_{n})^{-1} \colon TX^\bullet_{\gamma } \to TX^\bullet_n \to \colim_{n < \gamma} TX^\bullet_n$ where $\iota_n$ is the colimit inclusion.

The three split coequaliser equations are below.  Firstly
 $$\langle T(j^{\gamma}_{n}) \rangle \circ \iota_n \circ T(j^{\gamma}_{n})^{-1} =  T(j^{\gamma}_{n}) \circ T(j^{\gamma}_{n})^{-1}  = 1$$ where the first equality is by definition.  Secondly, we have $$\langle \eta_{X^\bullet_\gamma} \circ j_{n+1}^\gamma \circ x_n \rangle \circ  \iota_n \circ T(j^{\gamma}_{n})^{-1} = \eta_{X^\bullet_\gamma} \circ j_{n+1}^\gamma \circ x_n \circ T(j^{\gamma}_{n})^{-1}$$
again by definition.  Lastly, 
\begin{align*} j_{n+1}^\gamma \circ x_n \circ T(j^{\gamma}_{n})^{-1} \circ \eta_{X^\bullet_\gamma} & = j_{n+1}^\gamma \circ x_n \circ \eta_{X^\bullet_n} \circ (j^{\gamma}_{n})^{-1} & \textrm{(by naturality of $\eta$)}\\
    & = j_{n+1}^\gamma \circ j_n^{n+1} \circ (j^{\gamma}_{n})^{-1} & \textrm{(by definition of algebraic chains)}\\
    & = j^{\gamma}_{n} \circ (j^{\gamma}_{n})^{-1} = 1. & \textrm{(by functoriality of $j$)}\end{align*}
\end{proof}

\begin{proof}[Proof of Theorem \ref{thm:chains2}]
The proof closely follows that of Theorem \ref{thm:chains} with only a few small adaptations.  Let $X \in \mathcal C$.
\emph{(1) We start by showing that under either assumption, there exists a limit ordinal $\alpha$ such that $T_1$ and $T_2$ preserve the colimit of $(X^{\bullet}_n)_{n < \alpha}$}.  Indeed, suppose (b) and consider $T_1 \times T_2 \colon \mathcal C^2 \to  \mathcal C^2$.  The pointwise $(E,M)$ factorisation system on $ \mathcal C^2$ is clearly proper, since pointwise epis (resp. monos) are epi (resp. mono) and co-well-powered, since the set of pointwise $E$-quotients of an object in the product category is simply the product of the sets of $E$-quotients in the original category.  
Now $T_1 \times T_2$ then preserves colimits of pointwise $M$-chains of length $\alpha$, and so, as in Step 1 of Proposition \ref{thm:chains}, there exists a limit ordinal $\beta$ such that $T_1 \times T_2$ preserves the colimit of the $\beta$-chain $((X^{\star}_n),(X^{\star}_n)) _{n < \beta}$ in $ \mathcal C^2$.  But since colimits are pointwise in the product, therefore both $T_1$ and $T_2$ preserve the colimit $(X^{\star}_n)_{n < \beta}$.  Therefore, without loss of generality, we can assume in either case, that there exists a limit ordinal $\alpha$ such that both $T_1$ and $T_2$ preserve the colimit of $(X^{\star}_n)_{n < \alpha}$, in which case we will show that $X^{\star}$ stabilises at $\alpha$.

\emph{(1) We show that $j_\alpha^{\alpha +1}$ and $j_{\alpha+1}^{\alpha +2}$ are invertible.} The proof that $j_\alpha^{\alpha +1}$ is invertible is identical to that in Step 2 of Proposition \ref{thm:chains}, since the formulae for the coequalisers defining $X^\star$ and $X^\bullet$ coincide at a limit ordinal.  In order to prove that $j_{\alpha+1}^{\alpha +2}$ is invertible, observe that the triangle in the diagram below commutes.
\begin{equation}\label{eq:doublefork}
\xymatrix{
 T_1X^\star_\alpha   \ar@<1ex>[drr]^{T_1x_\alpha \circ T_1\eta_{X^\star_\alpha}}   \ar@<-1ex>[drr]_{T_1x_\alpha \circ \eta_{T_1X^\star_\alpha}} \\
 & &  T_1X^\star_{\alpha+1} \ar[drr]_{x_{\alpha+1}} \ar[rr]^{x_\alpha \circ T(j_\alpha^{\alpha+1})^{-1}} && X^\star_{\alpha+1} \ar@{-->}[d]^{j_{\alpha+1}^{\alpha+2}}\\
  T_2X^\star_\alpha \ar@<1ex>[urr]^{T_1x_\alpha \circ \lambda_{X^\star_\alpha}} \ar@<-1ex>[urr]_{T_1(j_\alpha^{\alpha+1}) \circ \gamma_{X^\star_\alpha}} 
 & & & & X^{\star}_{\alpha+2} }
 \end{equation}
 Since $x_{\alpha+1}$ is by definition the joint coequaliser, $j_{\alpha+1}^{\alpha+2}$ will be invertible just when the horizontal morphism $x_\alpha \circ T(j_\alpha^{\alpha+1})^{-1} \colon  T_1X^\star_{\alpha+1} \to X^\star_{\alpha+1}$ is a joint coequaliser.  Now this morphism is the split coequaliser of the upper parallel pair by the same argument as in Step 3 of Proposition \ref{thm:chains} --- therefore, it suffices to show that it coequalises the lower parallel pair (since being a coequaliser of one fork plus a cocone over the whole diagram trivially implies being the joint coequaliser of the whole diagram).
 
Firstly, observe that the lower path $x_\alpha \circ T(j_\alpha^{\alpha+1})^{-1} \circ T_1(j_\alpha^{\alpha+1}) \circ \gamma_{X^\star_\alpha}$ equals $x_\alpha \circ \gamma_{X^\star_\alpha}$ on cancelling inverses.  Therefore, it suffices to show that the upper path also equals $x_\alpha \circ \gamma_{X^\star_\alpha}$, which we show by precomposing with the colimit inclusions $T_2j^\alpha_n \colon T_2 X^\star_n \to T_2X^\star_\alpha$.
  
  \begin{align*}
        x_\alpha \circ T(j_\alpha^{\alpha+1})^{-1} \circ T_1x_\alpha \circ \lambda_{X^\star_\alpha} \circ T_2j^\alpha_n   & =  & (\textrm{by naturality of } \lambda) \\
        x_\alpha \circ T(j_\alpha^{\alpha+1})^{-1} \circ T_1x_\alpha \circ  T_1T_1 j^\alpha_n \circ \lambda_{X^\star_n}  & = &  (\textrm{by definition of an algebraic chain}) \\ 
        x_\alpha \circ T(j_\alpha^{\alpha+1})^{-1} \circ T_1 j^{\alpha+1}_{n+1} \circ T_1 x_n  \circ \lambda_{X^\star_n}  & =   & (\textrm{by functoriality of chains and cancelling inverses})\\ 
        x_\alpha \circ Tj_{n+1}^{\alpha}  \circ T_1 x_n  \circ \lambda_{X^\star_n}    & =  & (\textrm{by definition of an algebraic chain}) \\
        j^{\alpha+1}_{n+2} \circ x_{n+1} \circ T_1 x_n  \circ \lambda_{X^\star_n} & = & (\textrm{by definition of special algebraic chains}) \\ 
         j^{\alpha+1}_{n+2} \circ x_{n+1}\circ T_1(j^{n+1}_{n}) \circ \gamma_{X^\star_n}  &= &  (\textrm{by definition of an algebraic chain}) \\
         x_{\alpha} \circ T_1(j_{n+1}^{\alpha}) \circ T_1(j^{n+1}_{n}) \circ \gamma_{X^\star_n}  & = & (\textrm{by functoriality of algebraic chain}) \\ 
         x_{\alpha} \circ T_1(j_{n}^{\alpha}) \circ \gamma_{X^\star_n} & = &  (\textrm{by naturality of } \gamma ) \\
         x_{\alpha} \circ  \gamma_{X^\star_\alpha} \circ T_{2}j_{n}^{\alpha} &
    \end{align*}
 
\emph{(2) Next, we prove that if $j_{n}^{n+1}$ and  $j_{n+1}^{n+2}$ are invertible, then so is $j_{n+2}^{n+3}$}. To see this, observe that the morphism $j_{n+2}^{n+3}$ is the induced morphism between joint coequalisers induced by the natural transformation of diagrams
\[\xymatrix{
 T_1X^\star_n \ar[rrrr]^{T_1 j_{n}^{n+1}}  \ar@<1ex>[drr]^{T_1x_n \circ T_1\eta_{X^\star_n}}   \ar@<-1ex>[drr]_(.40){T_1x_n \circ \eta_{T_1X^\star_n}} &&&&  T_1X^\star_{n+1}  \ar@<1ex>[drr]^{T_1x_{n+1} \circ T_1\eta_{X^\star_{n+1}}}   \ar@<-1ex>[drr]_(.35){T_1x_{n+1} \circ \eta_{T_1X^\star_{n+1}}}\\
 & &  T_1X^\star_{n+1} \ar[rrrr]^(.30){T_1j_{n+1}^{n+2}}  &&&& T_1X^\star_{n+2} \\
  T_2X^\star_n \ar@<1ex>[urr]^{T_1x_{n} \circ \lambda_{X^\star_{n}}} \ar@<-1ex>[urr]_{T_1(j_n^{n+1}) \circ \gamma_{X^\star_{n}}} \ar[rrrr]_{T_2 j_{n}^{n+1}}  &&&& T_2X^\star_{n+1} \ar@<1ex>[urr]^(.35){T_1x_{n+1} \circ \lambda_{X^\star_{n+1}}} \ar@<-1ex>[urr]_(.55){T_1(j_{n+1}^{n+2}) \circ \gamma_{X^\star_{n+1}}}
 & & & & }
  \]
 whose components $j_{n}^{n+1}$ and  $j_{n+1}^{n+2}$ are by assumption invertible.  As such, the induced morphism $j_{n+2}^{n+3}$ is itself invertible.

\emph{(3) Finally we prove that $j^\beta_{\alpha}$ is invertible for all $\alpha \leq \beta.$}  We do so by transfinite induction.  The base case is established in (1) above so it remains to prove the inductive step.

Firstly, the case that $\beta$ is a limit ordinal is exactly as in Step 4 of Proposition \ref{thm:chains}.  Secondly, suppose $\beta$ is a successor.  There are three cases that we need to consider.  The case of $\beta = \gamma+1$, for $\gamma$ a limit ordinal, is exactly as in Step 4 of Proposition \ref{thm:chains}, since then the coequalisers defining $X^{\star}$ and $X^{\bullet}$ have the same form.   Next, we turn to the case $\beta = \gamma+2$, for $\gamma$ a limit ordinal.  By induction, it suffices to prove that $j_{\gamma+1}^{\gamma +2}$ is invertible.  The proof of this is identical to the proof in Step 3 that $j_{\alpha+1}^{\alpha +2}$ is invertible, with $\gamma$ substituted for $\alpha$, with the exception that to prove that the ``upper path'' in \eqref{eq:doublefork} has the required form, it suffices to precompose with any $T_2j^\alpha_n \colon T_2 X^\star_n \to T_2 X^\star_\gamma$ for $\alpha \leq n < \gamma$, since all of these are invertible.  The ensuing calculation then applies unchanged.

Finally, the  case in which $\beta = \gamma+3$ follows from Step 3 above using the inductive hypothesis.
\end{proof}

\end{document}